\documentclass{amsart}
\usepackage{amssymb,indentfirst,xspace,bm}
\usepackage{framed,paralist,suffix,color}
\usepackage{subfigure}
\usepackage{IEEEtrantools}
\usepackage{mathtools}
\usepackage[normalem]{ulem}

\newif\ifdraft
\draftfalse

\newif\ifTwoColumn
\TwoColumnfalse

\ifTwoColumn
  \usepackage[landscape,includehead,headheight=12pt,headsep=10pt,left=\columnsep,right=\columnsep,top=12.98999pt,bottom=\columnsep,]{geometry}
  \twocolumn
\else
  \ifdraft
    \usepackage[notref,notcite]{showkeys}
  \fi
\fi

\usepackage{cite}
\usepackage[colorlinks]{hyperref}
\definecolor{link1}{rgb}{0,0,.7}
\definecolor{link2}{rgb}{0,0.25,0.5}
\hypersetup{colorlinks,allcolors=link1,citecolor=link2}

\makeatletter
\@mparswitchfalse
\makeatother

\newcommand{\warning}[1]{\typeout{}\typeout{WARNING: #1 at line \the\inputlineno}\typeout{}}
\newenvironment{todo}[1][TODO]{%
    \ifdraft\else\warning{TODO still present in final version}\fi
    \MakeFramed{\advance\hsize-\width \FrameRestore}\textbf{#1. }}%
    {\endMakeFramed}
    {\end{todo}}

\makeatletter
\font\uwavefontb=lasyb10 scaled 652
\newcommand{\UWave}[2][blue]{\bgroup \markoverwith{\textcolor{#1}{\lower3.5\p@\hbox{\uwavefontb\char58}}}\ULon{#2}}
\newcommand{\SOut}[2][red]{\bgroup \markoverwith{\textcolor{#1}{\lower-.1ex\hbox{\uwavefontb\char58}}}\ULon{#2}}
\newcommand{\Highlight}[2][yellow]{\bgroup\markoverwith {\textcolor{#1}{\rule[-.2em]{2pt}{1.2em}}}\ULon{#2}}
\makeatother
\newcommand{\Remove}[1]{\SOut{#1}}
\newcommand{\Add}[1]{\UWave{#1}}

\newcommand{\remove}[1]{\ifmmode{\text{\Remove{\ensuremath{#1}}}}\else\Remove{#1}\fi}
\newcommand{\add}[1]{\ifmmode{\text{\Add{\ensuremath{#1}}}}\else\Add{#1}\fi}

\newcommand{\highlight}[2][yellow]{%
  \ifmmode\text{\Highlight[#1]{\ensuremath{#2}}}%
  \else\Highlight[#1]{#2}\fi}%
\newenvironment{beqn}[1][:C?s]%
    {\left\{\begin{IEEEeqnarraybox}[\relax][c]{#1}}%
    {\end{IEEEeqnarraybox}\right.}

  {\par\emph{#1\addpunct}}%
  {\par}

\newcommand{\del}{\partial}
\newcommand{\delt}{\del_t}
\newcommand{\delh}{\del_h}

\newcommand{\lap}{\triangle}

\newcommand{\inv}{^{-1}}

\newcommand{\grad}{\nabla}

\newcommand{\gradperp}{\grad^\perp}

\newcommand{\varmin}{\wedge}

\newcommand{\E}{\bm{E}}
\newcommand{\prob}{\bm{P}}
\newcommand{\mc}{\mathcal}
\newcommand{\AND}{\;\&\;}
\DeclareMathOperator{\erf}{erf}
\DeclareMathOperator{\sign}{sign}

\DeclareMathOperator{\diam}{diam}

\newcommand{\eps}{\varepsilon}
\renewcommand{\epsilon}{\eps}
\renewcommand{\leq}{\leqslant}
\renewcommand{\geq}{\geqslant}

\newcommand{\Chi}[1]{\Chi*{\{#1\}}}
\WithSuffix\newcommand\Chi*[1]{\chi_{\raise-.5ex\hbox{$\scriptstyle#1$}}}

\newcommand{\R}{\mathbb{R}}
\newcommand{\Z}{\mathbb{Z}}
\newcommand{\N}{\mathbb{N}}

\newif\iftextstyle
\textstyletrue
\everydisplay\expandafter{\the\everydisplay\textstylefalse}

\DeclarePairedDelimiter{\paren}{(}{)}
\DeclarePairedDelimiter{\brak}{[}{]}
\DeclarePairedDelimiter{\set}{\{}{\}}
\DeclarePairedDelimiter{\floor}{\lfloor}{\rfloor}

\DeclarePairedDelimiter{\abs}{\lvert}{\rvert}
\DeclarePairedDelimiter{\norm}{\lVert}{\rVert}

\newcommand{\st}[1][\big]{\;\iftextstyle|\else#1|\fi\;}
\newcommand{\given}[1][\big]{\;\iftextstyle|\else#1|\fi\;}
\WithSuffix\newcommand\given*{\;\middle|\;}
\newcommand{\at}[2][\Bigr]{\iftextstyle|\else#1|_{#2}\fi}
\WithSuffix\newcommand\at*[1]{\middle|_{#1}}

\newcommand{\defeq}{\stackrel{\scriptscriptstyle\text{def}}{=}}

\numberwithin{equation}{section}
\allowdisplaybreaks

\providecommand{\cites}[1]{\cite{#1}}

\newtheorem{theorem}{Theorem}[section]

\newtheorem{lemma}[theorem]{Lemma}

\newtheorem*{theorem*}{Theorem}
\newtheorem*{lemma*}{Lemma}
\newtheorem*{proposition*}{Proposition}
\newtheorem*{corollary*}{Corollary}

\theoremstyle{definition}

\theoremstyle{remark}

\newtheorem*{remark*}{Remark}

\newcommand{\eff}{_\text{eff}}

\begin{document}
  \title[Anomalous diffusion in fast cellular flows]
    {Anomalous diffusion  in fast cellular flows at intermediate time scales}
  \author{Gautam Iyer}
  \address{Department of Mathematical Sciences, Carnegie Mellon University, Pittsburgh PA 15213}
  \email{gautam@math.cmu.edu}

  \author{Alexei Novikov}
  \address{Department of Mathematics, Pennsylvania State University, State College PA 16802}
  \email{anovikov@math.psu.edu}


  \thanks{This material is based upon work partially supported by the National Science Foundation under grants
  DMS-0908011, 
  DMS-1007914, 
  DMS-1252912. 
  GI also acknowledges partial support from an Alfred P. Sloan research fellowship.
  The authors also thank the Center for Nonlinear Analysis (NSF Grants No. DMS-0405343 and DMS-0635983), where part of this research was carried out.}
  \begin{abstract}
    It is well known that on long time scales the behaviour of tracer particles diffusing in a cellular flow is effectively that of a Brownian motion.
    This paper studies the behaviour on ``intermediate'' time scales before diffusion sets in.
    Various heuristics suggest that an anomalous diffusive behaviour should be observed.
    We prove that the variance on intermediate time scales grows like $O(\sqrt{t})$.
    Hence, on these time scales the effective behaviour can not be purely diffusive, and is consistent with an anomalous diffusive behaviour.
  \end{abstract}
  \maketitle


  \section{Introduction}

  We study the behaviour of tracer particles diffusing in the presence of a strong array of opposing vortices (a.k.a. ``cellular flow'').
  Well known homogenization results show that on long time scales these particles effectively behave like a Brownian motion, with an enhanced diffusion coefficient (see for instance~\cites{Olla94,BensoussanLionsEtAl78,PavliotisStuart08}).
  On \emph{intermediate} time scales, however, tracer particles have their movement ``arrested'' in pockets of recirculation, leading to an anomalous diffusive behaviour~\cites{YoungPumirEtAl89,Young88,CardosoTabeling88,HaynesVanneste14,HaynesVanneste14b}.

  The purpose of this paper is to prove a quantitative estimate for the variance of these particles on intermediate time scales (Theorem~\ref{thmVarianceBound}).
  More precisely, we prove that the variance at time $t$ is $O(\sqrt{At})$, where $A$ is the P\'eclet number of the system.
  A purely diffusive process (e.g. Brownian motion) would have variance that is linear in $t$, and so the effective behaviour of the tracer particles at these time scales ``must be anomalous''.
  We remark, however, that we can not presently prove convergence of the particle trajectories to an effective process on intermediate time scales.

  
  \subsection{The long time behaviour.}\label{sxnHomog}
  We begin with a brief introduction to
  results about the long time behaviour of tracer particles.
  For concreteness, we model the position of the tracer particle by the SDE
  \begin{equation}\label{eqnXSDE}
    dX_t=-A v(X_t) \, dt + \sqrt{2} \, dW_t,	\qquad X_0=x.
  \end{equation}
  where $W$ is a 2D Brownian motion, $v$ is a velocity field with ``cellular'' trajectories, and $A > 0$ is the strength of the advection.
  We remark that $A$ is also the \emph{P\'eclet number} of this system, which is a non-dimensional parameter measuring the relative importance of cell size, velocity magnitude and the diffusion strength.

  For simplicity, we further assume
  \begin{equation} \label{eqnVH}
    v=\gradperp h \defeq \binom{-\del_2 h}{\phantom-\del_1 h},
    \quad\text{where }
    h(x_1, x_2) \defeq \sin(x_1) \, \sin(x_2).
  \end{equation}
  Geometrically, this is the velocity field associated with a two dimensional rectangular array of opposing vortices.
  The explicit choice of $v$ above is only to simplify many technicalities;
  the methods used (both the results we cite and in this paper) will apply to more general, but still cellular, velocity fields.

  The upshot of well known homogenization results is that the long time behaviour of $X$ is effectively that of a Brownian motion.
  More precisely, for any $\eps > 0$, define the process $X^\eps$ by $X^\epsilon_t = \eps X_{t/\eps^2}$.
  Then as $\eps \to 0$, the processes $X^\eps$ converges (in law) to $\sqrt{D\eff(A)} \, W'$, where $W'$ is a Brownian motion, and $D\eff(A)$ is the \emph{effective diffusivity} (see~\cites{BensoussanLionsEtAl78,PavliotisStuart08,Olla94}).

  The underlying mechanism is the interaction of two phenomena:
  The drift of the process $X^\eps$, which operates fast along closed orbits of size $\eps$, and the diffusion, which operates slowly moving $X^\eps$ between orbits.
  The combined effect produces an effective Brownian motion with an enhanced diffusion coefficient (see~\cite{FannjiangPapanicolaou94}).

  An outline of a rigorous proof (due to Freidlin~\cite{Fredlin64}) when $v$ is periodic proceeds as follows:
  Let the vector function $\chi$ be a periodic solution to the cell problem
  \begin{equation}\label{eqnChi}
    -\lap \chi + A v \cdot \grad \chi = - A v.
  \end{equation}
  It\^o's formula and elementary manipulations show
  \begin{equation*}
    X^\eps_t - X^\eps_0
      = -\eps\brak[\big]{
	    \chi(X_{\frac{t}{\eps^2}}) - \chi(X_0)
	  }
	+ \eps \int_0^{t/\eps^2} \sqrt{2} \paren[\big]{
	    I + \grad \chi(X_s)
	  } \, dW_s.
  \end{equation*}
  Since $\chi$ is independent of $\eps$, the drift term above converges to $0$ as $\eps \to 0$.
  By the ergodic theorem the quadratic variation of second term converges to $t D\eff(A)$, where
  \begin{equation*}
    \paren[\big]{ D\eff(A) }_{i,j}
      \defeq 2\delta_{i,j} + \frac{2}{\pi^2} \int_{(0, \pi)^2} \grad \chi_i \cdot \grad \chi_j.
  \end{equation*}
  L\'evy's criterion now shows that the limit is a Brownian motion with diffusion coefficient $\sqrt{D\eff(A)}$.
  We refer the reader to~\cites{Olla94,FannjiangPapanicolaou94} for more details.

  We remark further that the behaviour of $\chi$ and $D\eff$ have been extensively studied~\cites{Childress79,ChildressSoward89,FannjiangPapanicolaou94,Heinze03,Koralov04,NovikovPapanicolaouEtAl05,RosenbluthBerkEtAl87,Shraiman87,RhinesYoung83,SaguesHorsthemke86} as the P\'eclet number $A \to \infty$.
  It is well known that $D\eff \approx \sqrt{A} \, I$ asymptotically as $A \to \infty$.
  Consequently $\E^x \abs{X_t - x}^2 = O(\sqrt{A} \, t)$, when $A$ and $t$ are large.
  Recently Haynes and Vanneste~\cite{HaynesVanneste14} (see also~\cite{HaynesVanneste14b}) studied the long-time behavior of $X_t$ using  large deviations.
  They performed a formal asymptotic analysis in various regimes, and suggest that the ``active'' (or most mobile) tracer particles concentrate near the level set $\{h = 0\}$.
  The proof of our main results shows that a similar phenomenon occurs at intermediate time scales, and is described in Section~\ref{sxnIntroIntTime}.

  Finally, we also mention that the long time behaviour of $X$ has been studied under far more general assumptions on $v$ (see for instance~\cites{DolgopyatKoralov08,DolgopyatKoralov13,KaloshinDolgopyatEtAl05,DolgopyatFreidlinEtAl12}), and has a huge number of applications ranging from flame propagation to swimming (e.g.\ \cites{Taylor53,LiuXinEtAl11,NolenXinEtAl09,ShawThiffeaultEtAl07,Sowers06,ThiffeaultChildress10}).
  
  \subsection{The intermediate time behaviour.}\label{sxnIntroIntTime}

  In contrast to large time scales, the variance $\E^x \abs{X_t - x}^2$ at intermediate time scales doesn't grow linearly with time (see figure~\ref{fgrFitVar}).
  Theoretical and experimental results in~\cites{Young88,YoungPumirEtAl89,Young10,GuyonPomeauEtAl87} suggest instead
  \begin{equation}\label{eqnBillYoung}
    \E^x \abs{X_t - x}^2 = O(\sqrt{A} \sqrt{t}), \qquad\text{for } 1/A \ll t \ll  1, \text{ provided } h(X_0) = 0.
  \end{equation}
  The main contribution of this paper is to prove~\eqref{eqnBillYoung}, modulo a (necessary) logarithmic correction.
  \begin{figure}[htb]
    \newlength{\trajfigwidth}
    \setlength{\trajfigwidth}{.35\linewidth}
    \subfigure[\label{fgrFitVar}%
      Plot of the variance $\E \abs{X_t}^2$ vs $t$.
      The red dashed curve fits $(\E \abs{X_t}^2)^2$ to a linear function of $t$ on $\set{ t< 0.015 }$.
      The green dashed curve fits $\E \abs{X_t}^2$ to a linear function of $t$ on $t \geq 0.015$.%
    ]{%
      \includegraphics[width=1.2373113\trajfigwidth]
	{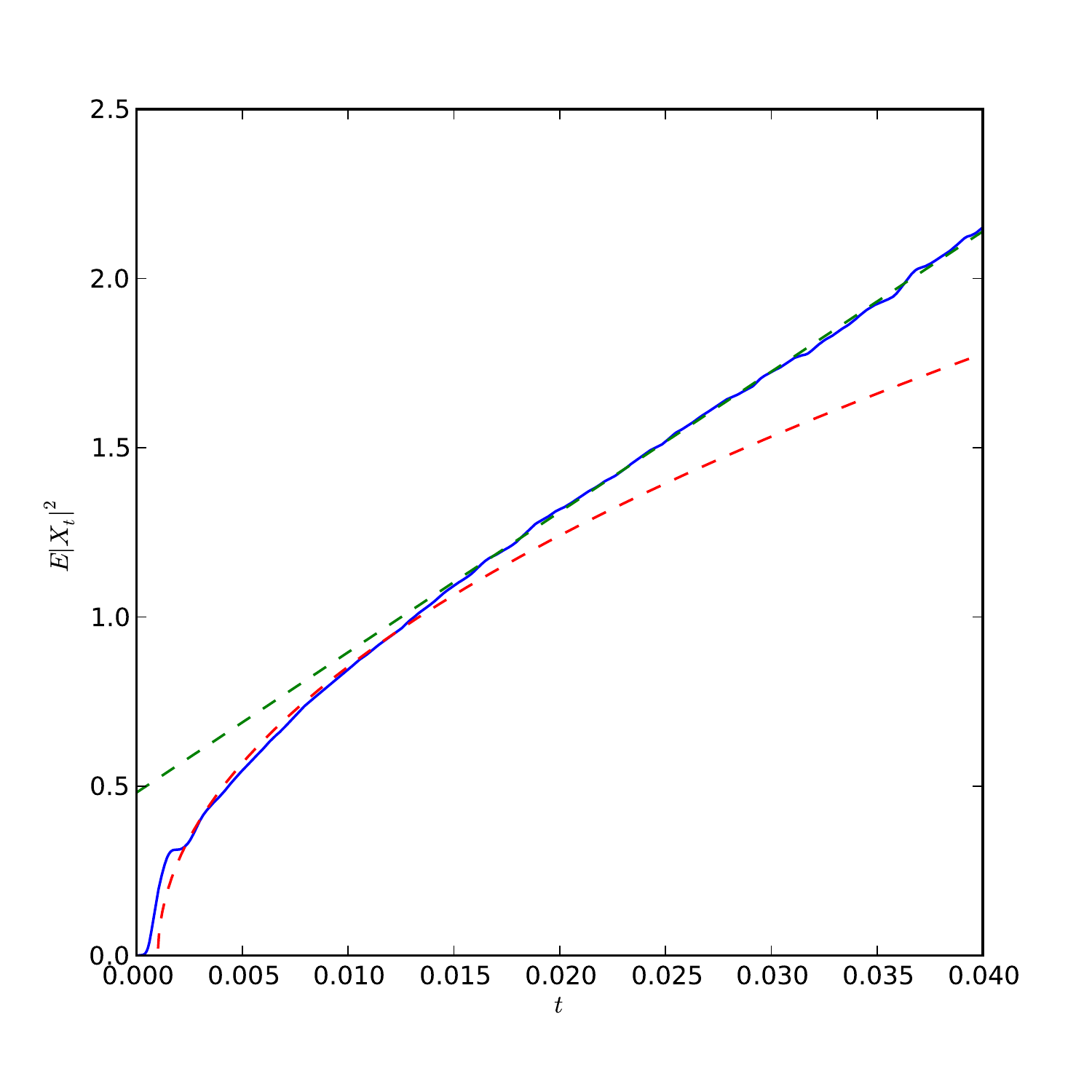}}
    \qquad
    \subfigure[\label{fgrTraj}%
      Three sample trajectories of $X$ for $0 \leq t \leq 2$.
      Cells that are nearly filled in correspond to long periods of rest.
      The remainder correspond to short periods of ballistic motion.%
    ]{%
      \raisebox{16.9pt}
	{\includegraphics[width=\trajfigwidth]{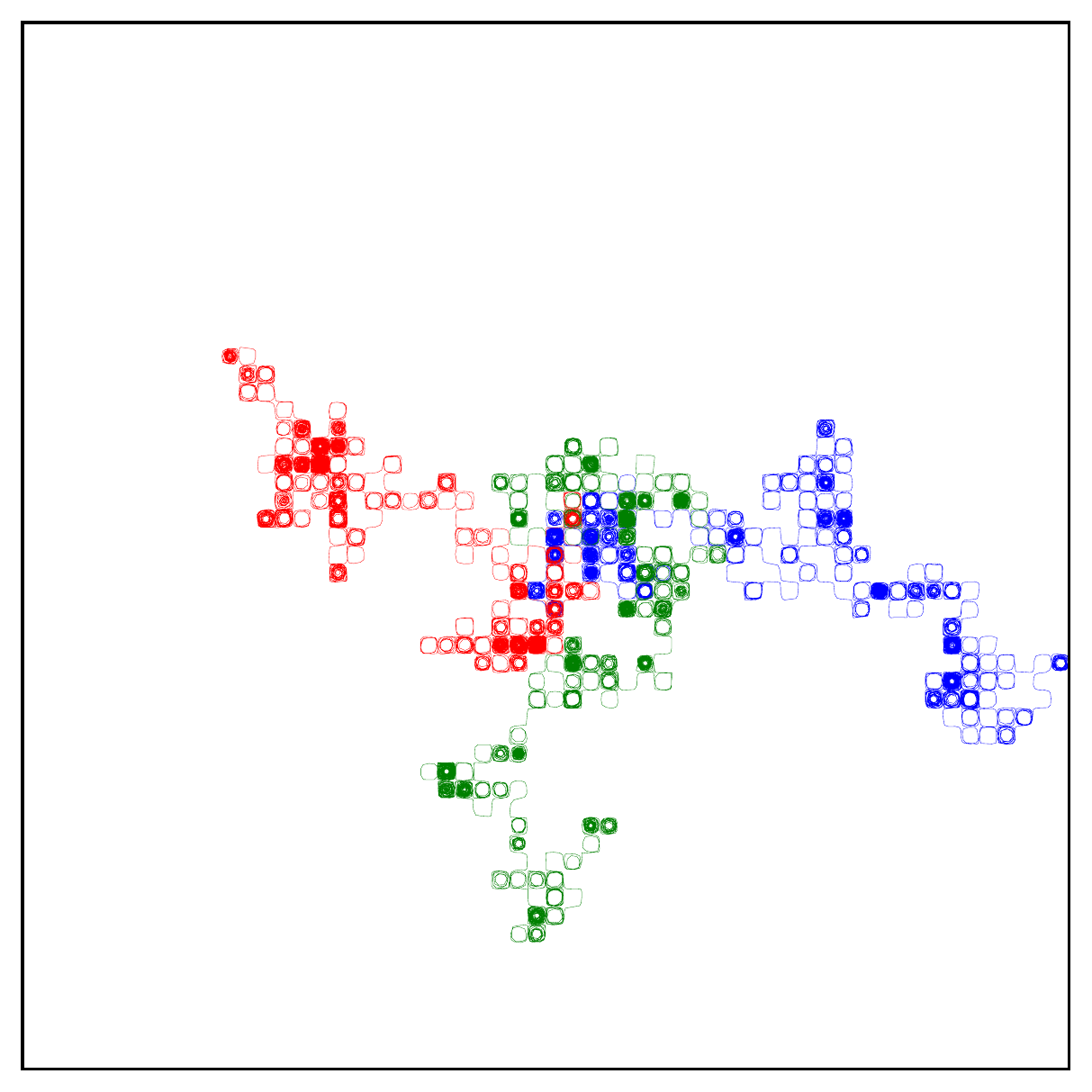}}%
     }
    \caption{Numerical simulations of equation~\eqref{eqnXSDE} with $A = 1000$.}
  \end{figure}

  \begin{theorem}\label{thmVarianceBound}
    Let $N > 0$, and define $\delta = N / \sqrt{A}$.
    There exists $T > 0$ and a positive constant $c$ such that whenever
    \begin{equation}\label{eqnTRange}
      \paren[\big]{ \delta \abs{\ln \delta} }^2 \ll t,
      \quad
      t \leq T
      \quad\text{and}\quad
      \delta \text{ is sufficiently small,}
    \end{equation}
    we have
    \begin{align}
      \label{eqnVarianceLowerBd}
      \inf_{ \abs{h(x)} < \delta } \E^x \abs{ X_t - x}^2
	&\geq
	\frac{ \sqrt{t} }{ c \delta \abs{\ln \delta}}
      \\
      \label{eqnVarianceUpperBd}
      \text{and}\qquad
      \sup_{ \abs{h(x)} < \delta } \E^x \abs{ X_t - x}^2
        &\leq \frac{c \sqrt{t}}{\delta}.
    \end{align}
  \end{theorem}
  
  We remark that both~\eqref{eqnBillYoung} and Theorem~\ref{thmVarianceBound} insist that trajectories start close to (or on) cell boundaries.
  This is essential for an anomalous diffusive effect to be observed, and will be explained later.
  Further, the exponent of $t$ appearing on the right of~\eqref{eqnBillYoung} depends on the boundary conditions used.
  The $\sqrt{t}$ growth was observed in~\cites{Young88,YoungPumirEtAl89,Young10,GuyonPomeauEtAl87}.
  In the case of a long strip with no slip boundary conditions on the velocity field, the papers~\cites{YoungPumirEtAl89,CardosoTabeling88} suggest that the variance grows like $t^{1/3}$ instead.

  We prove Theorem~\ref{thmVarianceBound} in Section~\ref{sxnMainProof}, and
  devote the remainder of this section to describing heuristics, the mechanism behind the proof.
  
  \subsubsection{A heuristic explanation.}
  Before delving into the technicalities of the proof, we provide a brief heuristic explanation suggested by W. Young~\cites{Young88,Young10}.
  The typical trajectory of $X$ spends most of its time trapped in cell interiors, which are pockets of recirculation (see figure~\ref{fgrTraj}).
  These particles are ``inert'' and contribute negligibly to the average travel distance.
  The largest contribution to the average travel distance is from the ballistic motion of a small fraction of ``active particles'' in a thin boundary layer around cell boundaries (see figure~\ref{fgrTracers}).

  This boundary layer should naturally be a region where the drift and diffusion balance each other~\cite{FannjiangPapanicolaou94}.
  Precisely, the time taken for the drift to transport a particle around the cell should be comparable to the time time taken for the noise to transport the particle across the boundary layer.
  This suggests that the boundary layer $\mc B_\delta$ should be defined by
  \begin{equation}\label{eqnChooseDelta}
    \mc B_\delta = \{ -\delta < h < \delta \},
    \quad\text{where}\quad
    \delta = \frac{N}{\sqrt{A}}
  \end{equation}
  and $N > 0$ is some constant.
  
  The distance travelled in a direction perpendicular to stream lines is influenced by the noise alone.
  Thus after time $t$, the variance of the perpendicular distance should be of order $t$.
  Hence the fraction of ``active particles'', i.e. particles that remain in $\mc B_\delta$, should be roughly $O( \delta / \sqrt{t} )$.

  \begin{figure}
    \subfigure[$t = .004$]
      {\includegraphics[width=.3\linewidth]{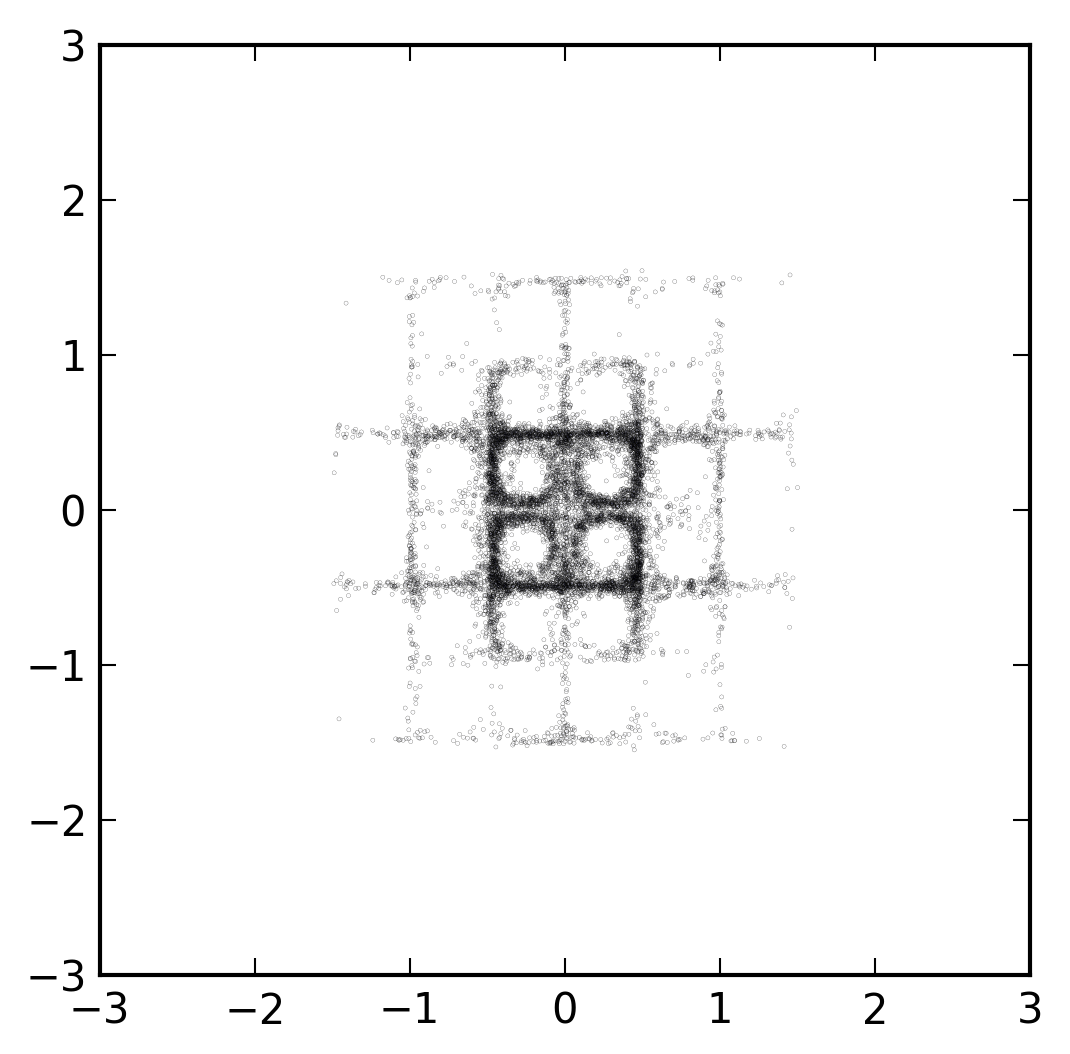}}
    \subfigure[$t = .012$]
      {\includegraphics[width=.3\linewidth]{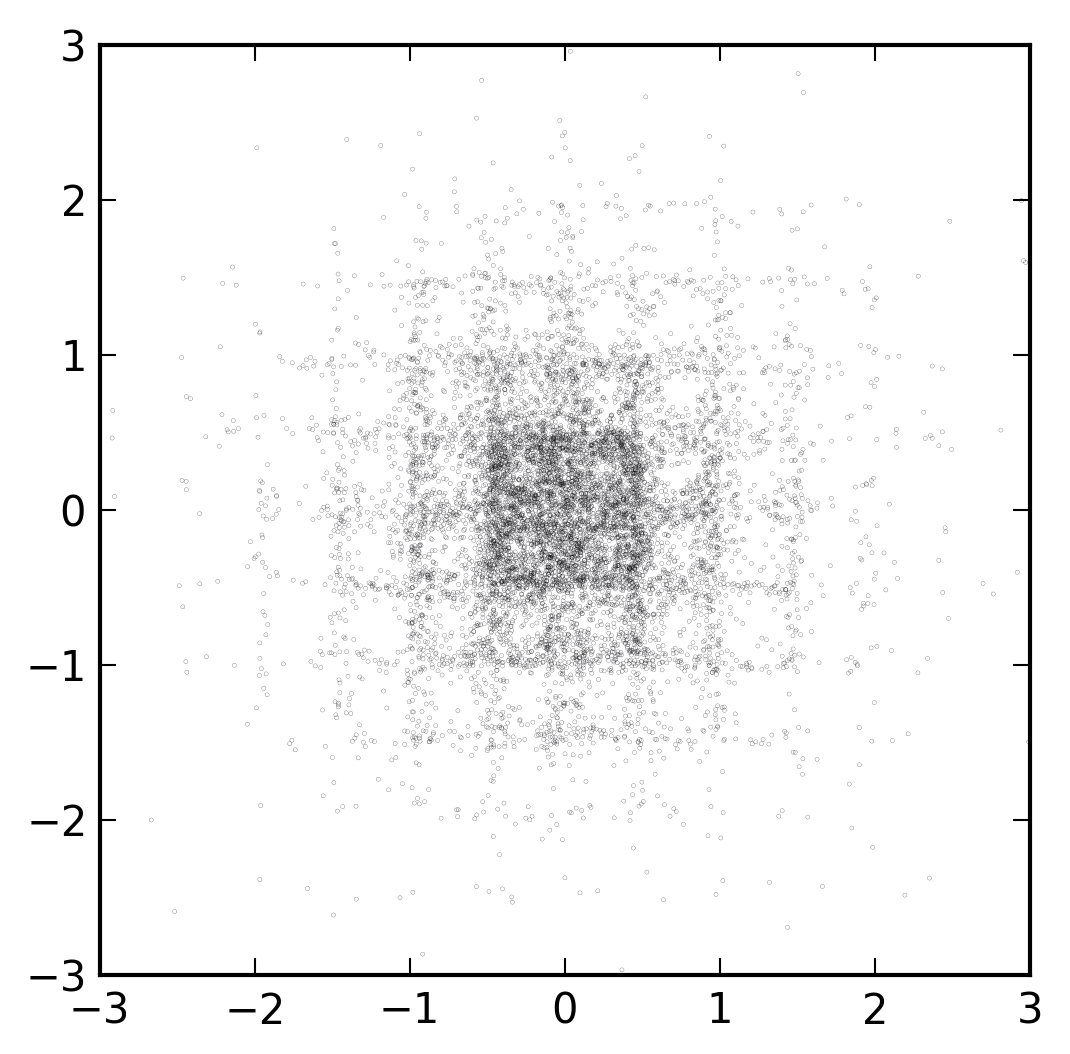}}
    \subfigure[$t = .040$]
      {\includegraphics[width=.3\linewidth]{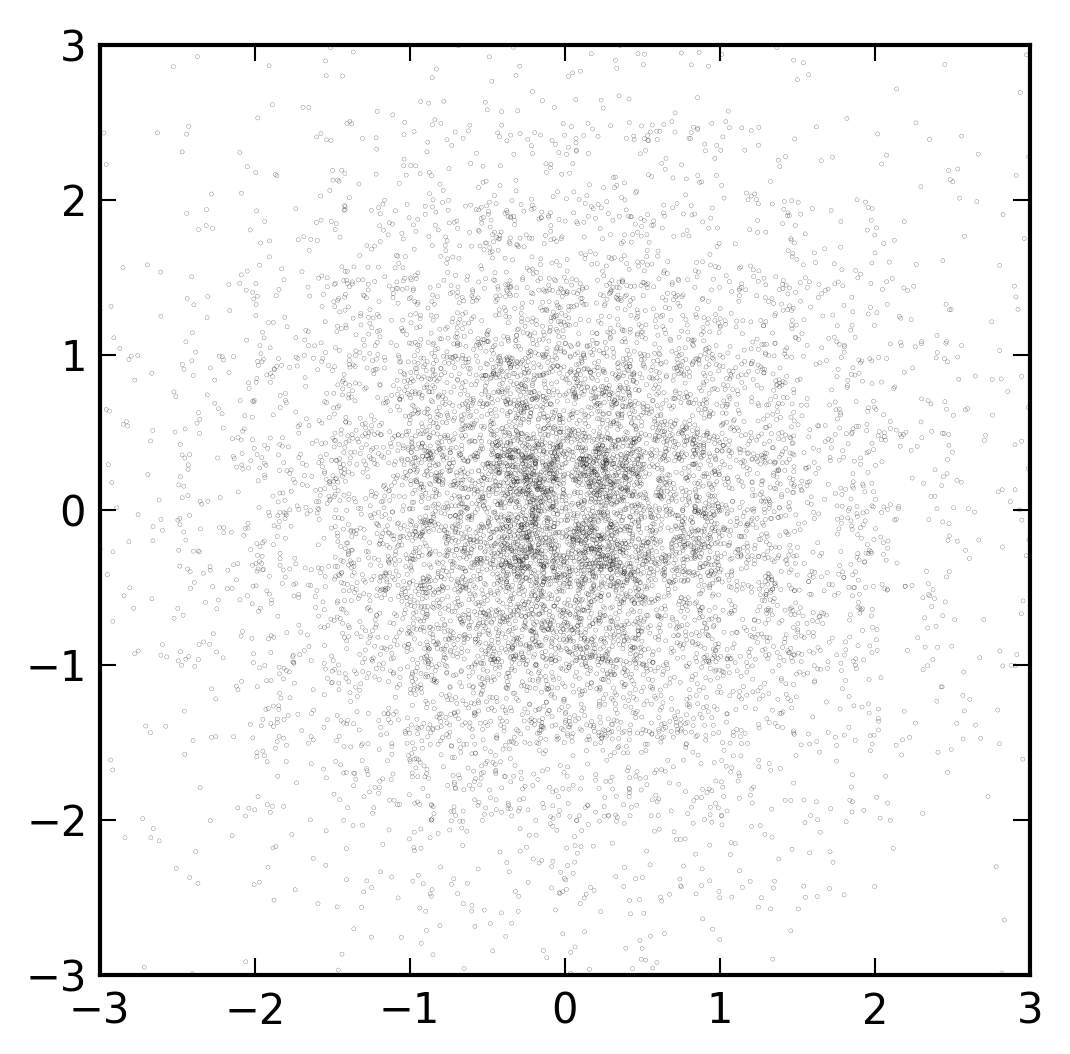}}
    \caption{%
      Numerical simulation of $10,000$ realizations of equation~\eqref{eqnXSDE} with $A = 1000$ and $X_0 = (0, 0)$.
      Initially most particles are ``active'' and travel ballistically near cell boundaries.
      As time increases these disperse into cell interiors becoming ``inert'' and the density approaches a Gaussian.
    }
    \label{fgrTracers}
  \end{figure}
  These ``active particles'' are advected along cell boundaries by the drift, which has magnitude $A$.
  They follow both the horizontal and the vertical cell boundaries in a manner akin to that of a random walk.
  Consequently, their behaviour after time $t$ should be that of a random walk that after $O(At)$ steps of size $O(1)$.
  Thus, the variance of the displacement of the ``active particles'' after time $t$ should be~$O(A t)$.
  Since the remaining particles travel negligible distances, and the fraction of ``active particles'' is $O( \delta /\sqrt{t} )$, the variance of the displacement travelled by \emph{all} particles should be $O( A t \cdot \delta / \sqrt{t}) $.
  This exactly gives~\eqref{eqnBillYoung}.

  The above estimate for the fraction of ``active particles'', and consequently the variance estimate in~\eqref{eqnBillYoung} is only expected to be valid on the time scales $1/A \ll t \ll 1$.
  Of course, for $t \gg 1$ the homogenized behaviour is observed (see also~\cite{Fannjiang02} for a more precise lower bound).
  Finally, we remark that an argument in~\cites{Young88,Young10} suggested that~\eqref{eqnBillYoung} is only valid for $A^{-2/3} \ll t \ll 1$.

  \subsubsection{The logarithmic slow-down and the idea behind a rigorous proof}

  Our approach to proving Theorem~\ref{thmVarianceBound} is by estimating the expected number of times the process $X$ crosses over the boundary layer $\mc B_\delta$.
  First, by solving a classical cell problem~\cites{FreidlinWentzell12,NovikovPapanicolaouEtAl05,FannjiangPapanicolaou94} one can show that that trajectories of $X$ starting on $\del \mc B_\delta$ exit the cell from each of the four edges with nearly equal probably.
  (In particular, trajectories don't directly exit from the closest edge, which is only a distance of $\delta$ away, with overwhelming probably.)
  Consequently, every time the process $X$ crosses $\mc B_\delta$ we expect it to have performed one independent $O(1)$-sized step of a random walk, and $\E^x \abs{X_t - x}^2$ should be comparable to the expected number of boundary layer crossings.

  Note, with this point of view the ``active particles'' from the previous section correspond to trajectories of $X$ that cross $\mc B_\delta$ more often.
  The expected number of boundary crossings will account for the fact that a large fraction of the particles are ``inert''.

  Since the convection is directed entirely in the tangential direction, crossing the boundary layer should be a purely diffusive effect.
  This suggests that the expected number of boundary crossings should be comparable to the expected number of crossings of Brownian motion over the interval $(-\delta, \delta)$.
  A standard calculation~\cite{KaratzasShreve91} shows that this is $O(\sqrt{t}/\delta)$, which immediately gives Theorem~\ref{thmVarianceBound}.

  The difficulty with proving this rigorously is a logarithmic slow down of trajectories near cell corners.
  %
  To elaborate, the typical trajectory of $X$ spends $O(\delta^2)$ time near cell edges where the Hamiltonian is non-degenerate.
  In this region $\mc B_\delta$ has width $\delta$ which can be crossed often by the diffusion alone on $O(\delta^2)$ time scales.
  However, typical trajectories of $X$ spend the much longer $O( \delta^2 \abs{\ln \delta} )$ time near cell corners where the Hamiltonian has a degenerate saddle point.
  This is problematic because in this region $\mc B_\delta$ has thickness $O(\sqrt{\delta})$, and diffusion alone will take too long to cross it unassisted 
  (see for instance~\cite{HaynesVanneste14b}).

  The reason our proof works is because even though trajectories of $X$ are too slow to cross $\mc B_\delta$ near cell corners, the drift moves them away from cell corners in time $O( \delta^2 \abs{\ln \delta})$.
  Once away, they will typically cross in $O( \delta^2 )$ time leading to a logarithmic slow down to~\eqref{eqnBillYoung}. 
  The meat of this paper is spent proving this by performing a delicate analysis of the behaviour in cell corners (Lemmas~\ref{lmaTau1BoundLower} and~\ref{lmaBoundaryExit}).

  We remark that our techniques don't presently show convergence of $X$ to an effective process on intermediate time scales, and we are working towards addressing this issue.
  A forthcoming result by Hairer, Koralov and Pajor-Gyulai~\cite{HairerKoralovEtAl14}, has a construction that might help identify the intermediate time process.
  Explicitly, in~\cite{HairerKoralovEtAl14} the authors rescale the domain, construct a time change that only increases when $X_t \in \mc B_\delta$, and identify both the law of the time change and the time changed process.
  Their proof, however, requires time to be large and does not work on intermediate time scales.

  \subsection{Plan of this paper.}
  
  In Section~\ref{sxnMainProof} we prove Theorem~\ref{thmVarianceBound}, modulo estimating the variance after each boundary layer crossing (Lemma~\ref{lmaVarNBounds}) and estimating the CDF of the boundary layer crossing times (Lemma~\ref{lmaTauINBound}).
  The key step in our proof of Lemma~\ref{lmaTauINBound} is obtaining a good estimate on the first crossing time over the boundary layer (Lemma~\ref{lmaTau1BoundLower}), and is done in Section~\ref{sxnFirstReturnMain}.
  This requires a delicate analysis near cell corners and forms the bulk of this paper.

  In Section~\ref{sxnNthReturn} we complete the proof of Lemma~\ref{lmaTauINBound}  by estimating the higher crossing times in terms of the first crossing time using the strong Markov property.
  Finally, in Section~\ref{sxnVarBound} we estimate the variance after each boundary layer crossing (Lemma~\ref{lmaVarNBounds}), which completes the proof of Theorem~\ref{thmVarianceBound}.

  \subsection*{Acknowledgements}
  We thank William R. Young for bringing our attention to the anomalous diffusive behaviour of tracer particles diffusing in a fast cellular flow.
  We also thank Martin Hairer, Leonid Koralov and Lenya Ryzhik for many stimulating discussions.
  Finally, we thank James T. Murphy III and Yue Pu for pointing out typographical errors in an early draft.

  \section{Proof of the main theorem.}\label{sxnMainProof}

  We devote this section to proving Theorem~\ref{thmVarianceBound}.
  This proof relies on two lemmas, which for clarity of presentation, we prove in subsequent sections.

  \begin{proof}[Proof of Theorem~\ref{thmVarianceBound}.]
    As explained earlier, the basic mechanism is that the process $X$ performs an independent $O(1)$-sized step of a random walk every time it crosses the boundary layer $\mc B_\delta$.
    We start by defining the boundary layer crossing times.
    Let  $\tau_0 = 0$, and recursively define the stopping times
    \begin{gather*}
      \sigma_n = \inf\{t \geq \tau_{n-1} \st X_t \not\in \mc B_\delta \}
      \quad\text{and}\quad
      \tau_n = \inf\{ t \geq \sigma_n \st h(X_t)=       0 \}.
    \end{gather*}
    We intuitively think of $\tau_n$ as the $n^\text{th}$ time $X$ hits the separatrix $\{h = 0\}$, and $\sigma_n$ as the first time after $\tau_n$ that $X$ emerges from the boundary layer $\mc B_\delta$.

    Given the symmetry of the advecting drift, it is convenient to deal with each coordinate of the flow separately.
    When convenient we will use the notation $X_i(t)$ to denote the $i^\text{th}$ coordinate of the flow $X$ at time $t$.
    For $i \in \set{1, 2}$ we define the ``coordinate'' crossing times as follows.
    Let  $\tau^i_0 = 0$, and recursively define
    \begin{gather*}
      \sigma^i_n = \inf\{t \geq \tau^i_{n-1} \st X_t \not\in \mc B_\delta \}
      \quad\text{and}\quad
      \tau^i_n
	= \inf\set{
	    \tau_k \st \tau_k > \tau^i_{n-1} \AND X_i(\tau_k) \in \pi \Z
	  }.
    \end{gather*}
    Intuitively, $\tau^i_n$ is the $n^\text{th}$ time the $i^\text{th}$ coordinate of $X$ hits the separatrix
    \begin{equation*}
      \set{h = 0} = (\R \times \pi \Z) \cup (\pi\Z \times \R).
    \end{equation*}
    Notice the sets $\{\tau^1_n\}$ and $\{\tau^2_n\}$ partition the set $\{\tau_n\}$, except on the null set where $X$ exits a cell exactly at a corner.

    Now we use an elementary telescoping sum to write the variance in terms of the boundary layer crossings.
    Namely, observe first
    \begin{equation*}
      \E^x \abs{ X(t) - x}^2
	= \E^x \abs{ X_1(t) - x_1}^2 +\E^x \abs{ X_2(t) - x_2}^2,
    \end{equation*}
    so it suffices to deal with each coordinate process individually.
    For $i \in \set{1, 2}$ notice
    \begin{align}
      \nonumber
      \E^x \abs{ X_i(t) - x_i}^2
	&= \E^x
	    \abs[\Big]{
	      \sum_{n=1}^{\infty}
		X_i(\tau^{i}_{n}\varmin t) - X_i(\tau^{i}_{n-1}\varmin t)
	    }^2
	  \\
      \label{eqnVar0}
	&=
	  \sum_{n=1}^{\infty}
	    \E^x
	      \abs[\big]{
		X_i(\tau^{i}_{n}\varmin t) - X_i(\tau^{i}_{n-1}\varmin t)
	      }^2.
    \end{align}
    Here we \emph{crucially} used the reflection symmetry of the drift $v$ to ensure that the cross terms in the last expression vanish.

    When $t < \tau^i_n$, the term inside the expectation in~\eqref{eqnVar0} vanishes.
    When $t > \tau^i_{n+1}$, we expect that this term should average to an $O(1)$ quantity.
    Thus each term on the right of~\eqref{eqnVar0} should be comparable to $\prob^x( \tau^i_n \leq t )$.
    We single this out as our first lemma:
    \begin{lemma}\label{lmaVarNBounds}
      There exists a positive constant $c$ such that if $i \in \{1, 2 \}$, $x \in \mc B_\delta$ and~\eqref{eqnTRange} holds, then
      \begin{equation}\label{eqnVarUpper}
	\E^x \abs{X_i(\tau^{i}_{n+1}\varmin t) - X_i(\tau^{i}_{n}\varmin t)}^2
	  \leq c \prob^x\paren{ \tau^i_n \leq t},
      \end{equation}
      and
      \begin{equation}\label{eqnVarLower}
        \E^x \abs{X_i(\tau^{i}_{n+1}\varmin t) - X_i(\tau^{i}_{n}\varmin t)}^2
        \geq
        \frac{1}{c} \prob^x \paren[\big]{ \tau^i_n \leq \frac{t}{2}}
      \end{equation}
    \end{lemma}

    The next step is to bound $\prob^x( \tau^i_n \leq t)$.
    Intuitively, $\tau^i_n$ should only depend on the movement of $X$ in a direction \emph{transverse} to the convection.
    Thus we should expect to bound $\prob^x( \tau^i_n \leq t)$ in terms of a purely diffusive process.
    Indeed, our next lemma is to show that $\prob^x( \tau^i_n \leq t)$ is comparable to that of Brownian motion.
    \begin{lemma}\label{lmaTauINBound}
      There exists a positive constant $c$ such that if $n \in \N$, $i \in \{1, 2 \}$, $x \in \mc B_\delta$ and~\eqref{eqnTRange} holds, then
      \begin{gather}
	\label{eqnTauINLower}
	\inf_{ \abs{h(x)} < \delta } \prob^x( \tau^i_n \leq t ) \geq
	  \paren[\Big]{1-  \frac{c n \delta \abs{\ln \delta} }{\sqrt{t}} }^+,
	\\
	\llap{\text{and}\qquad}
	\label{eqnTauINUpper}
	  \sup_{ \abs{h(x)} < \delta } \prob^x( \tau^i_n \leq t ) \leq
	    1-\erf  \paren[\Big]{ \frac{n \delta }{c \sqrt{t}} }.
      \end{gather}
    \end{lemma}

    Note that the right hand side of~\eqref{eqnTauINUpper} is exactly equal to the chance that a standard Brownian motion crosses the interval $(-\delta, \delta)$ at least $n$-times in time $t$.
    The right hand side of~\eqref{eqnTauINLower}, however, is much worse.
    It contains a $\abs{\ln \delta}$ factor, which is not merely a technical artifact, but present because of the logarithmic slow down of $X$ near the degenerate critical points of the Hamiltonian $h$.

    Momentarily postponing the proofs of Lemmas~\ref{lmaVarNBounds} and~\ref{lmaTauINBound}, we finish the proof of Theorem~\ref{thmVarianceBound}.
    Observe equation~\eqref{eqnVar0} and inequalities~\eqref{eqnVarUpper} and~\eqref{eqnVarLower} imply
    \begin{equation}\label{eqnVar1}
      \frac{1}{c}  \sum_{n=1}^\infty  \prob^x\paren[\big]{ \tau_n \leq \frac{t}{2} }
	\leq \E^x \abs{ X(t) - x}^2
	\leq c \sum_{n=1}^\infty  \prob^x ( \tau_n \leq t ).	
    \end{equation}
    Using Lemma~\ref{lmaTauINBound} both sides of the above can be estimated easily.

    Indeed, by~\eqref{eqnTauINUpper} we see
    \begin{equation}\label{eqnSumTauNUpper}
      \sum_{n=1}^\infty \prob^x( \tau^i_n \leq t )
	\leq
	  \sum _{n=1}^\infty
	    \paren[\Big]{
	      1 - \erf \paren[\big]{  \frac{n \delta }{c \sqrt{t}} }
	    }
	\leq \frac{c \sqrt{t}}{\delta}
    \end{equation}
    Here we used the convention that that $c > 0$ is a finite constant, independent of $A$, that may increase from line to line.

    For the lower bound, inequality~\eqref{eqnTauINLower} gives
    \begin{equation}\label{eqnSumTauNLower}
      \sum_{n=1}^\infty \prob^x \paren[\big]{ \tau^i_n \leq \frac{t}{2} }
	\geq \sum_{n=1}^\infty \paren[\Big]{ 1 - \frac{c n \delta \abs{\ln \delta} }{\sqrt{t}} }^+
	\geq \frac{1}{4} \floor[\Big]{ \frac{\sqrt{t}}{c \delta \abs{\ln \delta} } }
	\geq \frac{1}{8} \paren[\Big]{ \frac{\sqrt{t}}{c \delta \abs{\ln \delta} } }.
    \end{equation}
    The last inequality followed because of our assumption in~\eqref{eqnTRange} that guarantees $\delta \abs{\ln \delta} \ll \sqrt{t}$.

    Using~\eqref{eqnVar1}, \eqref{eqnSumTauNUpper} and~\eqref{eqnSumTauNLower} immediately yields both~\eqref{eqnVarianceLowerBd} and~\eqref{eqnVarianceUpperBd} as desired.
    This concludes the proof of Theorem~\ref{thmVarianceBound}, modulo the proofs of Lemmas~\ref{lmaVarNBounds} and~\ref{lmaTauINBound}.
  \end{proof}

  We prove Lemma~\ref{lmaTauINBound} first (in Section~\ref{sxnNthReturn}), as we use it in the proof of Lemma~\ref{lmaVarNBounds}.
  This is the key step in our paper.
  The hardest part in the proof is establishing~\eqref{eqnTauINLower} for $n = 1$, which we do in Section~\ref{sxnFirstReturnMain}.
  Finally, we prove Lemma~\ref{lmaVarNBounds} in Section~\ref{sxnVarBound} using both Lemma~\ref{lmaTauINBound} and ideas used in the proof of Lemma~\ref{lmaTauINBound}.

  \section{A lower bound for the first return time.}\label{sxnFirstReturnMain}

  We devote this section to the proof of the key step in Lemma~\ref{lmaTauINBound}:
  Namely we show that~\eqref{eqnTauINLower} holds for $n = 1$ and the stopping time~$\tau_1$ (Lemma~\ref{lmaTau1BoundLower} in Section~\ref{sxnFirstReturn}).
  This proof of this relies on two central lemmas: bounding the expected exit time from the boundary layer (Lemma~\ref{lmaBoundaryExit} in Section~\ref{sxnBoundaryExit}), and bounding the tail of the exit time from a cell (Lemma~\ref{lmaCellExit} in Section~\ref{sxnCellExit}).
  Both these Lemmas rely on estimating the chance that $X$ re-enters a corner (Lemma~\ref{lmaEdgeToCorner}) which we prove in Section~\ref{sxnCornerEntry}.

  \subsection{The first return to the separatrix.}\label{sxnFirstReturn}

  We devote this subsection to proving~\eqref{eqnTauINLower} for $n = 1$ and the stopping time $\tau_1$.
  For clarity, we state the result here as a lemma.

  \begin{lemma}\label{lmaTau1BoundLower}
    There exists a positive constant $c_0'$ such that if \eqref{eqnTRange} holds, then%
    \begin{equation}
      \label{eqnTau1BoundLowerExplicit}
      \inf_{\abs{h(x)} < \delta} \prob^x( \tau_1 \leq t )
	\geq
       \paren[\Big]{  1- \frac{c_0' \delta \abs{\ln \delta}}{\sqrt{t}} }^+,
    \end{equation}
  \end{lemma}

  The proof breaks up naturally into two steps.
  We recall $\tau_1$ is the first time $X$ hits the separatrix \emph{after} exiting the boundary layer.
  Thus to estimate $\tau_1$ we will first estimate the time $X$ takes to exit the boundary layer, and then estimate the time $X$ takes to return to the separatrix.

  As before, both these steps involve only the motion of $X$ across level sets of $h$, and should morally be independent of the convection term.
  There is, however, a logarithmic slow down near cell corners which introduces a logarithmic correction in our estimates.
  We state our results precisely below.


  \begin{lemma}\label{lmaBoundaryExit}
    There exists a constant $c$ such that when $\delta$ is sufficiently small
    \begin{equation}\label{eqnBoundaryExit}
      \sup_{x \in \mc B_\delta} \E^x \sigma_1
	\leq c \delta^2 \abs{\ln \delta}.
    \end{equation}
  \end{lemma}
  \begin{lemma}\label{lmaCellExit}
    Let $\tau = \inf\{ t \geq 0 \st h(X_t) = 0 \}$ be the hitting time of $X$ to the separatrix $\{h = 0\}$.
    There exists a positive constant $c$ such that if~\eqref{eqnTRange} holds, then
    \begin{equation}\label{eqnTauLower}
      \inf_{x \in \overline{\mc B}_\delta} \prob^x( \tau \leq t ) \geq 
	\paren[\Big]{ 1- \frac{c \delta \abs{\ln \delta} }{\sqrt{t}} }^+.
    \end{equation}
  \end{lemma}

  The proofs of both Lemmas~\ref{lmaBoundaryExit} and~\ref{lmaCellExit} are somewhat involved, and are the main ``technical content'' of this paper.
  If the Hamiltonian $h$ is non-degenerate, then one can easily show that the expected exit time of $X$ from $\mc B_\delta$ is comparable to $\delta^2$: the expected exit time of Brownian motion from the interval $(-\delta, \delta)$.

  In our case, however, the Hamiltonian $h$ is degenerate exactly at the cell corners.
  One wouldn't expect this to be problematic \emph{provided} $X$ did not spend too much time near the cell corners.
  Unfortunately, the process $X$ spends \emph{most} of the time near cell corners and compounds the problem.
  Precisely, when $X$ is in the boundary layer, it spends $O(\delta^2)$ time near cell edges (where $\grad h$ is non-degenerate) and $O(\delta^2 \abs{\ln \delta})$ near corners (where $\grad h$ degenerates).
  
  The estimate for $\tau$ is plagued with similar problems. 
  Further, the distribution of $\tau$ is heavy tailed and $\E \tau$ is much too large to be useful.
  Thus we are forced to take a somewhat indirect approach to Lemma~\ref{lmaCellExit}.
  We do this by estimating the Laplace transform of the CDF of $\tau$.

  Once Lemmas~\ref{lmaBoundaryExit} and~\ref{lmaCellExit} are established, however, Lemma~\ref{lmaTau1BoundLower} follows immediately.
  We present this below.
  \begin{proof}[Proof of Lemma~\ref{lmaTau1BoundLower}]
    Clearly
    \begin{equation*}
      \tau_1(x) = \sigma_1(x) + \tau( X_{\sigma_1(x)} ),
    \end{equation*}
    and so
    \begin{multline}\label{eqnPTau1GeqT}
      \prob^x\paren{ \tau_1 \leq t }
	\geq \prob\paren[\Big]{ \sigma_1(x) \leq \frac{t}{2} \text{ and } \tau( X_{\sigma_1(x)} ) \leq \frac{t}{2} }
	\\
	= \E^x \paren[\Big]{ \Chi{ \sigma_1 \leq t/2 } \E^{X_{\sigma_1}} \Chi{ \tau \leq t/2}  }
	\geq \prob^x \paren[\Big]{ \sigma_1 \leq \frac{t}{2} } \inf_{\abs{h(y)} = \delta } \prob^{y} \paren[\Big]{ \tau \leq \frac{t}{2}}.
    \end{multline}

    The second term on the right we can bound by Chebyshev's inequality and Lemma~\ref{lmaBoundaryExit}.
    Namely,
    \begin{equation*}
      \prob^x \paren[\big]{ \sigma_1 \leq \frac{t}{2} }
	\geq 1 - \frac{2}{t} \E^x \sigma_1
	\geq 1 - \frac{c \delta^2 \abs{\ln \delta}}{t}.
    \end{equation*}
    Thus Lemma~\ref{lmaCellExit} and inequality~\eqref{eqnPTau1GeqT} show
    \begin{equation*}
      \prob^x\paren{ \tau_1 \leq t }
	\geq
	  \paren[\Big]{ 1 - \frac{c \delta^2 \abs{\ln \delta}}{t} }
	  \paren[\Big]{  1- \frac{c \delta \abs{\ln \delta}}{\sqrt{t}} }
	\geq
	  1 - \frac{c \delta \abs{\ln \delta} }{\sqrt{t}} - \frac{c \delta^2 \abs{\ln \delta}}{t}.
    \end{equation*}
    By assumption~\eqref{eqnTRange}, the third term on the right can be absorbed into the second by increasing the constant $c$.
    This proves~\eqref{eqnTau1BoundLowerExplicit} as desired.
  \end{proof}

  \subsection{The exit time from the Boundary Layer.}\label{sxnBoundaryExit}
  
  In this subsection we aim to prove Lemma~\ref{lmaBoundaryExit}.
  As mentioned earlier, the main difficulty is that the process $X$ spends ``most'' of the time near cell corners where the Hamiltonian is degenerate.
  The main idea behind our proof is as follows: First, by constructing an explicit super-solution, we show that $X$ leaves the vicinity of cell corners in time $\delta^2 \abs{\ln \delta}$.
  Next, we show that the chance that $X$ ``re-enters'' a corner is bounded above by a constant $P_0 < 1$.
  Now using a geometric series argument we bound the expected exit time.

  To make this precise we need to introduce the natural action-angle coordinates associated to the Hamiltonian $h$.
  Recall the \emph{separatrix} is the set $\{h = 0\}$, and a \emph{cell} is a connected component of the complement of the separatrix.
  Fix a cell $Q_0$ with center $q_0 = (q_{0,1}, q_{0,2})$.
  Let $\theta$ be the solution of the PDE
  \begin{equation*}
    \begin{beqn}
      \grad \theta \cdot \grad h = 0
	& in $Q_0 - \{q_0 + (x, 0) \st x \geq 0\}$\\
      \theta(x_1, x_2)
	= \sign( h(q_0) )
	  \brak[\Big]{
	    \tan\inv\paren[\Big]{ \frac{x_2 - q_{0,2}}{x_1 - q_{0,1}} }
	    + \frac{\pi}{4}
	  }
	& on $\del Q_0$.
    \end{beqn}
  \end{equation*}
  The above boundary condition ensures that on streamlines of $v$, $\theta$ increases in the direction of $v$.
  Explicitly, $\theta$ increases in the counter-clockwise on cells where $h$ is positive, and in the clockwise on cells where $h$ is negative.
  Note further that on cell corners we have $\theta = n \pi / 2$ for $n \in \Z$.

  The map $x \mapsto (h, \theta)$ defines the natural action-angle coordinates local to each cell.
  For a pair of adjacent cells, we shift the angular coordinate in one cell by a multiple of $\pi$ to ensure continuity of this coordinate.
  By abuse of notation we still use $(h, \theta)$ to denote the local coordinates on a \emph{pair} of adjacent cells.
  This will be used repeatedly to obtain estimates along cell edges.

  Using the $(h, \theta)$ coordinates we define the ``corner'' and ``edge'' regions as follows.
  Fix $\beta_0 > 0$ to be some small constant, and define
  \begin{equation}\label{eqnCornerEdgeDef}
    \mc C \defeq
      \set[\big]{
	x \in \mc B_\delta \st
	\abs[\Big]{\theta(x) - \frac{n \pi}{2}} < \beta_0
	\text{ for } n \in \set{0, \dots, 3}
      },
    \quad\text{and}\quad
    \mc E \defeq \mc B_\delta - \overline{ \mc C }.
  \end{equation}
  Connected components of $\mc C$ are neighbourhoods of cell corners, and connected components of $\mc E$ are neighbourhoods of cell edges.
  Since the Hamiltonian $h$ is only degenerate in cell corners, we know $\abs{\grad h}$ and $\abs{\grad \theta}$ are bounded below away from $0$ on each connected component of $\mc E$.

  We are now ready to precisely state the lemmas required to prove Lemma~\ref{lmaBoundaryExit}.
  We begin with the time taken to exit the edge and corner regions.
  \begin{lemma}\label{lmaRhoE}
    Let $\rho_e$ be the first exit time of $X$ from the edge region $\mc E$.
    There exists a constant $c$ (independent of $\delta$) such that
    \begin{equation*}
      \norm{\E \rho_e}_{L^\infty(\mc E)} \leq c \delta^2.
    \end{equation*}
  \end{lemma}
  \begin{lemma}\label{lmaRhoC}
    Let $\beta_0' > \beta_0$ and define the fattened corner region
    \begin{equation}\label{eqnFatCornerEdgeDef}
      \mc C' \defeq
	\set[\big]{
	  x \in \mc B_\delta \st
	  \abs[\Big]{\theta(x) - \frac{n \pi}{2}} < \beta_0'
	  \text{ for } n \in \set{0, \dots, 3}
	},
      \quad\text{and}\quad
      \mc E' \defeq \mc B_\delta - \overline{ \mc C }'.
    \end{equation}
    Let $\rho_c$ be the first exit time of $X$ from the fattened corner region $\mc C'$.
    If $\beta_0, \beta_0'$ are sufficiently small then there exists a constant $c$ independent of $A$ such that
    \begin{equation*}
      \norm{\E \rho_c}_{L^\infty(\mc C')} \leq c \delta^2 \abs{\ln \delta}.
    \end{equation*}
  \end{lemma}

  Next, we state a lemma estimating probability that $X$ re-enters a corner.

  \begin{lemma}\label{lmaEdgeToCorner}
    Let $\rho_e$ be the first exit time of $X$ from $\mc E$.
    There exists a constant $P_0 = P_0( \beta_0, \beta_0', N)$ independent of $A$ such that for all $A$ sufficiently large we have
    \begin{equation}\label{eqnEdgeToCorner}
      \sup_{x \in \mc E'} \prob^x( X_{\rho_e} \in \del \mc E \cap \mc B_\delta ) \leq P_0
      \quad\text{and}\quad
      P_0 < 1.
    \end{equation}
  \end{lemma}

  We first show how Lemmas~\ref{lmaRhoE}--\ref{lmaEdgeToCorner} can be used to prove Lemma~\ref{lmaBoundaryExit}, and prove Lemmas~\ref{lmaRhoE}--\ref{lmaEdgeToCorner} subsequently.

  \begin{proof}[Proof of Lemma~\ref{lmaBoundaryExit}]
    Let $\rho_c$ be the first exit time of $X$ from the fattened corner $\mc C'$, and $\rho_e$ be the first exit time of $X$ from the edge $\mc E$.
    By the strong Markov property for any $x \in \mc C'$ we have
    \begin{multline*}
      \E^x \sigma_1
	= \E^x \Chi{ \rho_c < \sigma_1 } \sigma_1
	  + \E^x \Chi{\rho_c \geq \sigma_1} \sigma_1
	\\
	=
	  \E^x
	    \Chi{ \rho_c < \sigma_1 }
	    \paren[\big]{
	      \rho_c + \E^{X_{\rho_c}} \sigma_1
	    }
	  + \E^x \Chi{\rho_c \geq \sigma_1} \rho_c
	\leq
	  \sup_{y \in \mc E'} \E^y \sigma_1 + \E^x \rho_c
    \end{multline*}
    Similarly, for any $y \in \mc E'$ we have
    \begin{align*}
      \E^y \sigma_1
	&= \E^y \Chi{ \rho_e < \sigma_1 } \sigma_1
	  + \E^y \Chi{\rho_e \geq \sigma_1} \sigma_1
	\\
	&=
	  \E^y
	    \Chi{ \rho_e < \sigma_1 }
	    \paren[\big]{
	      \rho_e + \E^{X_{\rho_e}} \sigma_1
	    }
	  + \E^y \Chi{\rho_e \geq \sigma_1} \rho_e
	\\
	&\leq
	  \prob^y ( \rho_e < \sigma_1 )\;
	  \sup_{z \in \mc C} \E^z \sigma_1
	  + \E^y \rho_e
    \end{align*}

    Using Lemma~\ref{lmaEdgeToCorner} we note that $\prob^y( \rho_e < \sigma_1 ) \leq P_0$ for all $y \in \mc E'$.
    Thus
    \begin{equation*}
      \norm{\E \sigma_1}_{L^\infty(\mc B_\delta)}
	\leq
	  P_0 \norm{\E \sigma_1}_{L^\infty(\mc B_\delta)}
	  + \norm{\E \rho_e}_{L^\infty(\mc E)}
	  + \norm{\E \rho_c}_{L^\infty(\mc C')}.
    \end{equation*}
    Consequently
    \begin{equation}\label{eqnSigma1Bd}
      \norm{\E \sigma_1}_{L^\infty(\mc B_\delta)}
	\leq (1 - P_0)\inv
	  \paren[\big]{ 
	    \norm{\E \rho_e}_{L^\infty(\mc E)}
	    + \norm{\E \rho_c}_{L^\infty(\mc C')}
	  }.
    \end{equation}
    Using Lemma~\ref{lmaRhoE} and~\ref{lmaRhoC} this immediately implies~\eqref{eqnBoundaryExit} as desired.
    This completes the proof of Lemma~\ref{lmaBoundaryExit}, modulo Lemmas~\ref{lmaRhoE}--\ref{lmaEdgeToCorner}.
  \end{proof}

  It remains to prove Lemmas~\ref{lmaRhoE}--\ref{lmaEdgeToCorner}.
  We prove Lemma~\ref{lmaEdgeToCorner} first (subsection~\ref{sxnCornerEntry}), as the result will be re-used in later sections.
  Finally we conclude this subsection with the proof of Lemmas~\ref{lmaRhoE} and~\ref{lmaRhoC} (in subsection~\ref{sxnCornerExit}).

  \subsubsection{The corner entry probability}\label{sxnCornerEntry}

  In this subsection we prove Lemma~\ref{lmaEdgeToCorner}.
  The main idea in the proof is that width of the boundary layer is chosen so that the convection in the $\theta$ direction and the diffusion in the $h$ direction balance each other.
  The diffusion in the $\theta$ direction, however, is an order of magnitude smaller, and can be neglected.
  Consequently, we bound the chance of entering a corner using the solution to a parabolic problem in $h$ and $\theta$.

  \begin{proof}[Proof of Lemma~\ref{lmaEdgeToCorner}]
    To prove the Lemma it suffices to restrict our attention any connected component of $\mc E$.
    Let $\mc E_0$ be one such component.
    Assume, for simplicity, that the angular coordinate in $\mc E_0$ varies between $\beta_0$ and $\pi/2 - \beta_0$.

    Let $\zeta_e(x) = \prob^x( X_{\rho_e} \in \del \mc E_0 \cap \mc B_\delta )$.
    Clearly $\zeta_e = \zeta_1 + \zeta_2$, where
    \begin{equation*}
      \zeta_1(x) = \prob^x\paren[\big]{ X_{\rho_e} \in \set{\theta = \beta_0} }
      \quad\text{and}\quad
      \zeta_2(x) =
	\prob^x\paren[\big]{
	  X_{\rho_e} \in \set[\big]{\theta = \frac{\pi}{2} - \beta_0}
	}.
    \end{equation*}
    We bound $\zeta_1$ and $\zeta_2$ by constructing super-solutions to the associated PDE's.

    Since the bound for $\zeta_2$ is simpler, we address it first.
    Note that $\zeta_2$ is the chance that a particle travels directly against the drift to exit $\mc E_0$.
    This is highly unlikely and we will show that $\zeta_2$ decays to $0$ exponentially with $A$.
    To prove this, note that~$\zeta_2$ satisfies
    \begin{equation}\label{eqnZeta2}
      \begin{beqn}
	-\lap \zeta_2 + A v \cdot \grad \zeta_2 = 0 & in $\mc E_0$,\\
	\zeta_2 = 0
	  & on $\del \mc E_0 - \set[\big]{ \theta = \frac{\pi}{2} - \beta_0 }$
	\\
	\zeta_2 = 1
	  & on $\del \mc E_0 \cap \set[\big]{ \theta = \frac{\pi}{2} - \beta_0 }$.
      \end{beqn}
    \end{equation}
    We construct a super-solution to this equation by choosing
    \begin{equation*}
      \bar \zeta_2 =
	\exp\paren[\big]{
	  \gamma_2 A \paren[\big]{ \theta - \frac{\pi}{2} + \beta_0 }
	}
    \end{equation*}
    where $\gamma_2 > 0$ is a constant that will be chosen later.

    To verify $\bar \zeta_2$ is a super-solution to~\eqref{eqnZeta2} we use the identities
    \begin{equation}\label{eqnLapHTheta}
      \lap =
	\abs{\grad \theta}^2 \del_\theta^2
	+ \abs{\grad h}^2 \del_h^2
	+ \lap \theta \del_\theta
	+ \lap h \del_h
      \quad\text{and}\quad
      v \cdot \grad = \abs{\grad \theta} \abs{\grad h} \del_\theta
    \end{equation}
    to compute
    \begin{equation*}
      -\lap \bar \zeta_2 + A v \cdot \grad \bar \zeta_2
	= -\abs{\grad \theta}^2 \del_\theta^2 \bar\zeta_2
	  + \paren{
	      A \abs{\grad \theta} \abs{\grad h}
	      - \lap \theta
	    } \del_\theta \bar\zeta_2.
    \end{equation*}
    Since
    \begin{equation*}
      \del_\theta \bar \zeta_2 > 0,
      \quad
      \del_\theta^2 \bar \zeta_2 > 0,
      \quad\text{and}\quad
      \inf_{\mc E_0} \; \abs{\grad h} \abs{\grad \theta} > 0,
    \end{equation*}
    there exists a constant $\alpha_0 = \alpha_0( N, \beta_0 )$ such that
    \begin{equation*}
      -\lap \bar \zeta_2 + A v \cdot \grad \bar \zeta_2
	\geq
	  \abs{\grad \theta}^2
	    \paren{
	      -\del_\theta^2 \bar \zeta_2
	      + \frac{A}{\alpha_0} \del_\theta \bar \zeta_2
	    }.
    \end{equation*}
    Choosing $\gamma_2 = 1 / \alpha_0$ makes the right hand side of the above vanish.
    Further, since $\bar \zeta_2 \geq \zeta_2$ on $\del \mc E_0$, the maximum principle implies $\bar \zeta_2 \geq \zeta_2$ on all of $\mc E_0$.
    In particular
    \begin{equation*}
      \sup_{\mc E' \cap \mc E_0} \zeta_2
	\leq \sup_{\mc E' \cap \mc E_0} \bar \zeta_2
	= \exp( -\gamma_2 A (\beta_0' - \beta_0) ),
    \end{equation*}
    which converges to $0$ exponentially with $A$.
    \medskip

    Now we turn to bounding $\zeta_1$.
    We recall that the width of the boundary layer is chosen so that the convection in the $\theta$ direction and the diffusion in the $h$ direction balance each other.
    The diffusion in the $\theta$ direction, however, is an order of magnitude smaller, and can be neglected.
    This is the main idea in our proof, and we bound $\zeta_1$ from above by using the solution to a parabolic problem.

    To construct an upper bound, we first observe that $\zeta_1$ satisfies
    \begin{equation}\label{eqnZeta1}
      \begin{beqn}
	-\lap \zeta_1 + A v \cdot \grad \zeta_1 = 0 & in $\mc E_0$,\\
	\zeta_1 = 0
	  & on $\del \mc E_0 - \set[\big]{ \theta = \beta_0 }$
	\\
	\zeta_1 = 1
	  & on $\del \mc E_0 \cap \set[\big]{ \theta = \beta_0 }$.
      \end{beqn}
    \end{equation}
    We will find a function $\bar \zeta_1$ which is a super-solution to~\eqref{eqnZeta1}, and is of the form
    \begin{equation*}
      \bar\zeta_1(h, \theta) \defeq
	  \xi( \sqrt{A} h, \theta )
          + \alpha_1' (\delta^2 - h^2).
    \end{equation*}
    The function $\xi$ and constant $\alpha_1'$ above will be chosen later.
    
    For convenience, we define the rescaled coordinate $h' = \sqrt{A} h$.
    Using~\eqref{eqnLapHTheta} and the identity $\lap h = -2h$ we compute
    \begin{align}
      \nonumber
      \MoveEqLeft
      -\lap \bar\zeta_1  + A v \cdot \grad \bar\zeta_1
      \\
      \nonumber
      &= -\lap \theta \, \partial_\theta \bar\zeta_1
	+ 2 h \, \partial_h \bar\zeta_1
	- \abs{\grad \theta}^2 \del_\theta^2 \bar\zeta_1
	- \abs{\grad h}^2 \delh^2 \bar\zeta_1
	+ A \abs{\grad h} \abs{\grad \theta} \del_\theta \bar\zeta_1
      \\
      \label{eqnBarZeta11}
      &=
	A \abs{\grad h}^2 \paren[\bigg]{
	  \paren[\Big]{
	    \frac{\abs{\grad \theta}}{\abs{\grad h}}
	      - \frac{\lap \theta}{A \abs{\grad h}^2}
	  } \del_\theta \xi
	  - \del_{h'}^2 \xi
	  }
	+ 2h' \del_{h'} \xi
	- \abs{\grad \theta}^2 \del_\theta^2 \xi
	\\
	\nonumber
	&\qquad
	  + \alpha_1'\paren[\big]{
	    2\abs{\grad h}^2 - 4h^2
	  }
	  .
    \end{align}
    Since $h, \theta \in C^2(\overline{\mc E}_0)$ and $\abs{\grad h} \neq 0$ in $\overline{ \mc E }_0$ we can find a finite constant $\alpha_1$ so that
    \begin{equation*}
      \sup_{\mc E_0}
	\paren[\Big]{
	  \frac{\abs{\grad \theta}}{\abs{\grad h}}
	    - \frac{\lap \theta}{A \abs{\grad h}^2}
	}
	\leq \alpha_1
    \end{equation*}
    for all $A$ sufficiently large.
    Further, we will ensure that the function $\xi$ is chosen so that $\del_\theta \xi < 0$.
    Consequently~\eqref{eqnBarZeta11} reduces to
    \begin{multline}\label{eqnBarZeta12}
      -\lap \bar\zeta_1  + A v \cdot \grad \bar\zeta_1
	\geq A \abs{\grad h}^2 \paren[\big]{
	  \alpha_1 \del_\theta \xi - \del_{h'}^2 \xi
	}
	\\
	+ 2h' \del_{h'} \xi
	- \abs{\grad \theta}^2 \del_\theta^2 \xi
	+ \alpha_1'\paren[\big]{
	  2\abs{\grad h}^2 - 4h^2
	}
	.
    \end{multline}

    Now, we choose $\xi$ to be the solution of the heat equation
    \begin{equation*}
      \begin{beqn}
	\alpha_1 \del_\theta \xi - \del_{h'}^2 \xi= 0
	  & for $h' \in (-2N, 2N)
	    \AND \theta \in (\beta_0, \frac{\pi}{2} - \beta_0)$
	\\
	\xi(h', \theta) = 0 & for $h' = \pm 2N$,
      \end{beqn}
    \end{equation*}
    with smooth concave initial data initial data such that
    \begin{equation*}
      \xi( h', 0 ) = 1 \quad \text{when } \abs{h'} < N,
      \quad\text{and}\quad
      \xi( \pm 2N ) = 0.
    \end{equation*}

    Observe that the boundary conditions for $\xi$ and the concavity of the initial data imply $\del_\theta \xi < 0$, which was used in the derivation of~\eqref{eqnBarZeta12}.
    Now~\eqref{eqnBarZeta12} simplifies to
    \begin{multline*}
      -\lap \bar\zeta_1  + A v \cdot \grad \bar\zeta_1
	\geq
	  2h' \del_{h'} \xi
	  - \abs{\grad \theta}^2 \del_\theta^2 \xi
	  + \alpha_1'\paren[\big]{
	    2\abs{\grad h}^2 - 4\delta^2
	  }
	\\
	\geq
	  -c \sup_{\abs{h'} \leq N}
	    \paren[\big]{ \abs{\del_{h'} \xi} + \abs{\del_\theta^2 \xi} }
	  + \alpha_1'\paren[\big]{
	    2\inf_{\mc E_0} \abs{\grad h}^2 - 4\delta^2
	  }
    \end{multline*}
    for some constant $c$, independent of $A$.
    Since the equation for $\xi$ and the domain are independent of $A$,  all bounds on $\xi$ are also independent of $A$.
    Thus, when $A$ is sufficiently large, we can choose $\alpha_1'$ large enough to ensure that we can ensure that the right hand side of the above is positive.
    Since $\bar \zeta_1 \geq \zeta_1$ on $\del \mc E_0$, we have shown that $\bar \zeta_1$ is a super-solution to~\eqref{eqnZeta1}.
    \medskip

    Finally, given the bounds on $\zeta_1$ and $\zeta_2$, we deduce~\eqref{eqnEdgeToCorner}.
    By symmetry of the flow we see
    \begin{equation}\label{eqnZetaFinal}
      \sup_{x \in \mc E'} \prob^x( X_{\rho_e} \in \del \mc E \cap \mc B_\delta )
      = \sup_{x \in \mc E' \cap \mc E_0} \zeta(x)
      \leq
	  P_1'
	  + \alpha_1 \delta^2
	  + \exp( -\gamma_2 A (\beta_0' - \beta_0) ).
    \end{equation}
    where
    \begin{equation*}
      P_1' = \sup_{x \in \mc E' \cap \mc E_0}\xi(x).
    \end{equation*}
    By the strong maximum principle, we know $P_1' < 1$.
    Also, as $\xi$ is independent of $A$, the constant $P_1'$ must also be so.
    Since the last two terms on the right of~\eqref{eqnZetaFinal} vanish as $A \to \infty$, the proof is complete.
  \end{proof}
  \subsubsection{The exit time from edges and corners}\label{sxnCornerExit}

  It remains to estimate the expected exit time from edges (Lemmas~\ref{lmaRhoE})  and from corners (Lemma~\ref{lmaRhoC}).
  The expected exit time from edges is quick, and we present it first.

  \begin{proof}[Proof of Lemma~\ref{lmaRhoE}.]
    Let $\varphi_e(x) = \E^x \rho_e$.
    We know that $\varphi_e$ satisfies the Poisson equation
    \begin{equation*}
      \begin{beqn}
	-\lap \varphi_e + A v\cdot \grad \varphi_e = 1
	  & in $\mc E$,
	\\
	\varphi_e = 0
	  & on $\del \mc E$.
      \end{beqn}
    \end{equation*}
    We claim $\varphi_e \leq \bar \varphi_e$, where $\bar \varphi_e = \alpha (\delta^2 - h^2)$ for some constant $\alpha$ to be chosen later.

    To see this, we use~\eqref{eqnLapHTheta} and compute
    \begin{equation*}
      -\lap \bar \varphi_e + A v \cdot \grad \bar \varphi_e
	= -\abs{\grad h}^2 \delh^2 \bar \varphi_e
	  - (\lap h) \delh \bar \varphi_e
	\geq \alpha \paren[\big]{ 2 \abs{\grad h}^2  - 4 \delta^2}
    \end{equation*}
    Since $\grad h$ is not degenerate in $\mc E$, we can choose $\alpha$ large enough so that
    \begin{equation*}
      -\lap \bar \varphi_e + A v \cdot \grad \bar \varphi_e \geq 1.
    \end{equation*}
    Clearly $\bar \varphi_e \geq \varphi_e$ on $\del \mc E$, thus the maximum principle implies $\bar \varphi_e \geq \varphi_e$ on all of $\mc E$.
    This immediately gives the desired bound on $\E \rho_e$ completing the proof.
  \end{proof}

  The proof of Lemma~\ref{lmaRhoC} requires a little more work, and we address it next.

  \begin{proof}[Proof of Lemma~\ref{lmaRhoC}]
    We know that $\varphi_c = \E^x \rho_c$ satisfies the PDE
    \begin{equation}\label{eqnVarphiC}
      \begin{beqn}
	-\lap \varphi_c + A v\cdot \grad \varphi_c = 1
	  & in $\mc C'$,
	\\
	\varphi_c = 0
	  & on $\del \mc C'$.
      \end{beqn}
    \end{equation}
    The main idea in this proof is to find a super solution of~\eqref{eqnVarphiC} that depends on only one coordinate.

    Without loss of generality we restrict our attention to $\mc C'_0$, the connected component of $\mc C'$ that contains the origin.
    We will find $\bar\varphi_c =\bar\varphi_c(x_1)$ such that
    \begin{equation}\label{eqnRhoBar}
      \begin{beqn}
        - \bar\varphi_c'' - A \sin(x_1) \cos(x_2) \,  \bar\varphi_c'  \geq 1
	    & in $\mc C_0'$,\\
        \varphi_c \geq  0 & on $\del \mc C_0'$.
      \end{beqn}
    \end{equation}

    To construct $\bar\varphi_c$, let $\bar x_1 = 3/ (2\sqrt{A})$ and define
    \begin{equation*}
      \bar \varphi_c(x) \defeq
	\begin{dcases}
	  g_0(\abs{x_1})  & \abs{x_1} \leq \bar x_1,\\
	  g_1(\abs{x_1})  & \abs{x_1} > \bar x_1.
	\end{dcases}
    \end{equation*}
    Here $g_0$ and $g_1$ are functions that have the following properties:
    For $x \in \mc C_0'$ with $x_1 \geq \bar x_1$ we require that the function $g_1$ satisfies
    \begin{equation}\label{eqnG1}
      -g_1'' - \frac{A x_1 g_1'}{2} \geq 1,
      \quad
      g_1' \leq 0,
      \quad\text{and}\quad
      g_1 \geq 0.
    \end{equation}
    When $0 \leq x_1 \leq \bar x_1$, we require that the function $g_0$ satisfies
    \begin{equation}\label{eqnG0}
      -g_0'' - A x_1 \abs{ g_0' } \geq 1,
      \quad
      g_0'(0)=0,
      \quad
      g_0'(\bar x_1 ) \geq g_1'(\bar x_1 )
      \quad\text{and}\quad
      g_0(\bar x_1) = g_1( \bar x_1 ).
    \end{equation}
    The equations~\eqref{eqnG1} and~\eqref{eqnG0} immediately guarantee that $\bar\varphi_c$ satisfies~\eqref{eqnRhoBar}.

    To estimate $\bar \varphi_c$ we find the functions $g_0$ and $g_1$ explicitly.
    Define
    \begin{equation*}
      \gamma_0 = \sup \{ x_1 \st x \in \mc C_0' \},
      \quad\text{and}\quad
      g_1(x_1) = \frac{18}{A} \ln \paren[\big]{ \frac{\gamma_0}{x_1} }.
    \end{equation*}
    The second and third inequality in~\eqref{eqnG1} are clearly satisfied.
    For the first inequality, we observe that if $\beta_0'$ is small enough we can guarantee $\cos(x_2) \geq 1/2$ in the region $\mc C_0'$.
    Consequently
    \begin{equation*}
      -g_1'' - A \sin(x_1) \cos(x_2) \,  g_1'
      \geq    
	-g_1'' - \frac{A x_1 \,  g_1'}{2}
      = -\frac{18}{A x_1^2} + \frac{18}{2}
      \geq -8+9 =1,
    \end{equation*}
    showing the first inequality in~\eqref{eqnG1}.

    For $g_0$, we define
    \begin{equation*}
      g_0(x_1) = a_0 - \frac{x_1^2}{2},
    \end{equation*}
    where the constant $a_0$ is chosen so that $g_0(\bar x_1) = g_1( \bar x_1 )$.
    Explicitly,
    \begin{equation*}
      a_0
	= g_1(\bar x_1) + \frac{\bar x_1^2}{2}
	= \frac{9}{8 A} + \frac{18}{A}
	  \ln \paren[\big]{ \frac{ 2\gamma_0 \sqrt{A} }{3} }
	\approx \frac{c  \ln A }{A}.
    \end{equation*}
    Now we compute
    \begin{equation*}
      -g_0'' - A \sin(x_1) \cos(x_2) \,  g_0'
      \geq -g_0''
      = 1,
    \end{equation*}
    giving the first inequality in~\eqref{eqnG0}.
    Clearly $g_0' = 0$ and
    \begin{equation*}
      g_0'( \bar x_1 )
	= -2 \bar x_1 
	= \frac{-3}{\sqrt{A}}
	\geq \frac{-12}{\sqrt{A}}
	= \frac{-18}{A \bar x_1}
	= g_1'(\bar x_1).
    \end{equation*}
    This establishes all the inequalities in~\eqref{eqnG0} and shows that $\bar \varphi_c$ indeed satisfies~\eqref{eqnRhoBar}.

    By the maximum principle $\varphi_c \leq \bar \varphi_c$ on all of $\mc C_0'$.
    Consequently
    \begin{equation*}
      \norm{\varphi_c}_{L^\infty(\mc C_0')} \leq 
      \norm{\bar \varphi_c}_{L^\infty(\mc C_0')}
      = a_0 \leq \frac{c \ln A}{A},
    \end{equation*}
    finishing the proof.
  \end{proof}

  \subsection{Tail bounds on the cell exit time.}\label{sxnCellExit}

  We devote this section to proving Lemma~\ref{lmaCellExit} showing that the tail of CDF of the exit time from a cell is bounded below by that of the passage time of Brownian motion.
  This proof is a little more technical than the proof of Lemma~\ref{lmaBoundaryExit}, mainly because $\tau$ is heavy tailed.
  The main idea is to obtain~\eqref{eqnTauLower} indirectly by estimating the Laplace transform.

  \begin{proof}[Proof of Lemma~\ref{lmaCellExit}]
    Let $Q_0$ be a cell containing the point $x$, and define $\psi(x, t) = 1 - \prob^x( \tau_1 \leq t )$.
    We know that $\psi$ satisfies the PDE
    \begin{equation}\label{eqnPsi}
      \begin{beqn}
	\delt \psi + A v \cdot \grad \psi - \lap \psi = 0 & in $Q_0$,\\
	\psi = 1 & on $Q_0 \times \{0\}$,\\
	\psi = 0 & on $\del Q_0 \times [0, \infty)$.
      \end{beqn}
    \end{equation}

    Clearly a bound of the form
    \begin{equation}\label{eqnPsiBound}
      \sup_{x \in \mc B_\delta^+} \psi(x,t)
	\leq    \frac{ c \delta \abs{\ln \delta}}{\sqrt{t} } 
    \end{equation}
    for all $t$ satisfying~\eqref{eqnTRange} is enough to complete the proof of Lemma~\ref{lmaCellExit}.
    Here $\mc B_\delta^+ \defeq \mc B_\delta \cap \set{ h > 0 }$, where we assume for simplicity that $h > 0$ in $Q_0$.
    As usual we assume that $c > 0$ is a finite constant, independent of $A$, that may increase from line to line.

    We prove~\eqref{eqnPsiBound} by bounding the Laplace transform $\varphi$, defined by%
    \begin{equation*}
      \varphi(x,\lambda) = \int_0^\infty e^{-\lambda t} \psi(x, t) \, dt.
    \end{equation*}
    Since $\psi$ is a non-negative decreasing function of time, observe
    \begin{equation*}
      \varphi(x, \lambda)
	\geq \int_0^{1/\lambda} e^{-\lambda t} \psi( x, t) \, dt
	\geq
	  \frac{1}{\lambda}
	  \psi\paren[\big]{ x, \frac{1}{\lambda} }
	  \paren[\big]{ 1 - \frac{1}{e} }.
    \end{equation*}
    Choosing $\lambda = 1/t$ gives
    \begin{equation*}
      t \psi(x,t)\leq  \frac{e}{e - 1} \varphi \paren[\big]{ x, \frac{1}{t} }.
    \end{equation*}
    Thus inequality~\eqref{eqnPsiBound} (and consequently Lemma~\ref{lmaCellExit}) will follow from an inequality of the form
    \begin{equation}\label{eqnMainGoal}
      \varphi(x,\lambda)
	\leq  \frac{c \delta \abs{\ln \delta}}{ \sqrt{\lambda} } 
      \quad
      \text{for } 1 \ll \lambda \ll A,
      \text{ and all }
      x \in \mc B_\delta^+.
    \end{equation}

    Observe $\varphi$ satisfies
    \begin{equation}\label{eqnPhi}
      \begin{beqn}
	- \lap \varphi  + A v \cdot \grad \varphi+ \lambda \varphi = 1 & in $Q_0$,\\
	\varphi = 0 & on $\del Q_0 $.
      \end{beqn}
    \end{equation}
    We aim to obtain an upper bound for $\varphi$ by constructing an appropriate super solution.
    For this construction, we will need two auxiliary functions: $\varphi_e$, and $\varphi_c$, which we define below.
    Roughly speaking, $\varphi_e$ will provide a good estimate near cell edges, and $\varphi_c$ will handle the corners.
    
    First we define $\varphi_e$ by
    \begin{equation}\label{eqnPhiE}
      \varphi_e
	= \alpha h(x) \left( \frac{2}{ \sqrt{\alpha\lambda}} - h(x) \right),
    \quad
    \end{equation}
    where $\alpha$ is a fixed constant that will be chosen later.

    To define $\varphi_c$, we need to ``fatten'' the boundary layer a little.
    Namely define $\eps = 1 / \sqrt{\alpha \lambda}$ and let $\varphi_c$ be the solution of
    \begin{equation}\label{eqnPhi2}
      \begin{beqn}
	- \lap \varphi_c  + A v \cdot \grad \varphi_c+ \lambda \varphi_c = \Chi*{\mc C_\eps^+} & in $\mc B_{\eps}^+$,\\
	\varphi_c = 0 & on $\del \mc B_{\eps}^+ $.
      \end{beqn}
    \end{equation}
    Here $\mc C_\eps^+ \defeq \mc C_\eps \cap \set {h > 0}$, where $\mc C_\eps$ the neighbourhood of the corners defined in~\eqref{eqnCornerEdgeDef}.
    For clarity of presentation in this section, we subscript our edge and corner regions with $\eps$ to indicate that their thickness is $\eps = O(1 / \sqrt{\lambda})$ and not $\delta = O(1 / \sqrt{A})$ as we have in other sections of this paper.

    We claim
    \begin{equation}\label{eqnVarphiSupSol}
      \varphi \leq \varphi_e + \varphi_c \quad\text{in } \mc B_\eps^+.
    \end{equation}
    To see this, observe first that the maximum principle guarantees $\varphi \leq 1/\lambda$.
    Thus on $\{ h = \eps \}$, we have $\varphi_e=1/\lambda$.
    Since both $\varphi_e$ and $\varphi_c$ are non-negative, we must have $\varphi \leq \varphi_e + \varphi_c$ on $\del \mc B_\eps^+$.

    On the interior of $\mc B_\eps^+$, we use~\eqref{eqnLapHTheta} and the identity $-\lap h = 2h$ to obtain
    \begin{multline*}
      - \lap \varphi_e + A v \cdot \grad \varphi_e+ \lambda \varphi_e
	= - \del_h \varphi_e \lap h - \del_h^2 \varphi_e \abs{\grad h}^2 + \lambda \varphi_e\\
	= 2 \alpha \abs{\grad h}^2 + 4 \alpha h (\eps - h) + \lambda \alpha h ( 2\eps - h )
	\geq 2 \alpha \abs{\grad h}^2
    \end{multline*}
    Let $\mc E_\eps^+ = \mc E_\eps \cap \set{h > 0}$, where $\mc E_\eps$ is the neighbourhood of the edges defined in~\eqref{eqnCornerEdgeDef}.
    Since $\grad h$ is non-degenerate in $\mc E_\eps^+$, we can choose $\alpha$ large enough so that $2 \alpha \abs{\grad h}^2 \geq 1$ on $\mc E_\eps^+$.

    Consequently, with this choice of $\alpha$,
    \begin{equation*}
      (-\lap + A v \cdot \grad + \lambda) ( \varphi_e + \varphi_c ) \geq 1 \quad\text{in } \mc B_\eps^+,
      \quad\text{and}\quad
      \varphi_e + \varphi_c \geq \frac{1}{\lambda} \geq \varphi \quad\text{on } \del \mc B_\eps^+.
    \end{equation*}
    Now~\eqref{eqnVarphiSupSol} follows immediately from the maximum principle.

    Once~\eqref{eqnVarphiSupSol} is established, we only need to control $\varphi_c$ appropriately in order to prove~\eqref{eqnMainGoal}.
    Since this is the heart of the matter and involves a delicate analysis of the behaviour near the degenerate corners, and we single it out as a lemma and momentarily postpone its proof.
    \begin{lemma}\label{lmaEtaMainBound}
      With $\varphi_c$ as above, we have
      \begin{equation}\label{eqnEtaMainBound}
	\varphi_c(x,\lambda)
	  \leq c \abs{\ln \delta} \paren[\Big]{
	    \delta^2 + \frac{ h(x) }{\sqrt{\lambda}}
	  }
	\quad\text{for all } x \in \mc B_\eps^+.
      \end{equation}
    \end{lemma}
    Observe that~\eqref{eqnEtaMainBound} immediately implies
    \begin{equation*}
      \varphi(x)
	\leq \varphi_e(x,\lambda) + \varphi_c(x, \lambda)
	\leq
	  c \paren[\Big]{
	      \frac{h(x)}{\sqrt{\lambda}}
	      + \abs{\ln \delta} \paren[\Big]{
		  \delta^2 + \frac{ h(x)}{\sqrt{\lambda}}
		}
	    }
    \end{equation*}
    which yields~\eqref{eqnMainGoal} when $1 \ll \lambda \ll A$ and $\delta$ is sufficiently small.
    This concludes the proof of Lemma~\ref{lmaTau1BoundLower}, modulo the proof of Lemma~\ref{lmaEtaMainBound}.
  \end{proof}

  \subsubsection*{The corner estimate.}

  To finish the proof of Lemma~\ref{lmaCellExit} we need to prove Lemma~\ref{lmaEtaMainBound}.
  Unfortunately, in this situation we can not use a geometric series argument based on Lemma~\ref{lmaEdgeToCorner}, because we need to work with boundary layers of thickness much larger than $O(\delta)$.
  In this case the constant $P_0$ in Lemma~\ref{lmaEdgeToCorner} degenerates to $1$ and our geometric series argument will not work directly.

  In order to make such an argument work, one needs to  is to unfold~$\varphi_c$ to the universal cover.
  This effectively replaces the geometric series argument based on Lemma~\ref{lmaEdgeToCorner} with a finer version that makes better use of the distance of the initial position to the separatrix.
  The proof, however, is a lot more involved and we chose not to follow this approach here.
  We instead estimate $\varphi_c$ by explicitly constructing a super-solution which is very concave in the $\theta$ direction in cell corners as described below.

  \begin{proof}[Proof of Lemma~\ref{lmaEtaMainBound}] 
    Let $h_0 = \delta^2$.
    Observe that the function
    \begin{equation*}
      \psi = 2 h \abs{ \ln h }
    \end{equation*}
    is a super solution to equation~\eqref{eqnPhi2}.
    Thus~\eqref{eqnEtaMainBound} certainly holds for $h \leq h_0$.

    In the region $h \geq h_0$, we use the ansatz 
    \begin{equation*}
      \bar \varphi_c(h, \theta) = \alpha_1 \varphi_e(h) + g(\theta),
    \end{equation*}
    where $ \varphi_e(h)$ solves~\eqref{eqnPhiE}, and $\alpha_1$ is a constant, and $g$ is a periodic function that only depends on $\theta$.
    We will see later that $\alpha_1 \approx \abs{\ln \delta}$ will be chosen later.

    To find an equation for $g$ we compute
    \begin{multline}\label{eqnVarphiC1}
      -\lap \bar \varphi_c + A v \cdot \grad \bar \varphi_c
	  + \lambda \bar \varphi_c
	\geq
	  \alpha_1 \Chi*{\mc E_\eps^+}
	  -\abs{\grad \theta}^2 g''
	  + \paren[\big]{ A \abs{\grad \theta} \abs{\grad h} - \lap \theta} g'
	\\
	=
	  \alpha_1 \Chi*{\mc E_\eps^+}
	  + \abs{\grad \theta}^2 \paren[\big]{
	    -g'' +
	      \paren[\Big]{ \frac{ A \abs{\grad h} }{\abs{\grad \theta}}
		-\frac{ \lap \theta }{\abs{\grad \theta}^2} } g'
	    }.
    \end{multline}
    Here we used the fact that $g \geq 0$, which we will be guaranteed by our construction of $g$.
    A direct calculation shows that there exist constants $b_0$, $b_1$ such that
    \begin{equation}\label{eqnB0B1}
      0 < A b_0
	\leq
	  \paren[\Big]{ \frac{ A \abs{\grad h} }{\abs{\grad \theta}}
	    -\frac{ \lap \theta }{\abs{\grad \theta}^2} }
	\leq
	  A b_1 < \infty
    \end{equation}
    holds on $\mc B_\eps^+$.

    The quickest way to verify~\eqref{eqnB0B1} is to observe that in a neighbourhood of a corner the coordinates $(h, \theta)$ are asymptotically
    \begin{equation*}
      h \approx x y,
      \quad\text{and}\quad
      \theta \approx x^2 - y^2
    \end{equation*}
    up to constants.
    Consequently,
    \begin{equation}\label{eqnGradTheta}
      \abs{\grad \theta}^2 \approx \sqrt{h^2 + \theta^2} 
      \quad\text{and}\quad
      \lap \theta \leq O(\abs{\grad \theta}^2).
    \end{equation}
    Since $\grad \theta$ is non-degenerate away from corners, the existence of $b_0$ and $b_1$ satisfying~\eqref{eqnB0B1} follows.%
    \footnote{
      An alternate explicit (but cumbersome) proof of~\eqref{eqnB0B1} can be obtained by a direct computation using the explicit choice of the angular coordinate $\theta = \cos x \sec y$ near cell corners.
      In this case we see
      \begin{equation*}
	  \frac{\lap \theta}{\abs{\grad \theta}^2} =
	    \frac{2 \sin^2 y \cos x \cos y}
	      {\sin^2 x + \sin^2 y - 2 \sin^2 x \sin^2 y}.
      \end{equation*}
      from which~\eqref{eqnB0B1} quickly follows.
    }

    Returning to~\eqref{eqnVarphiC1} we see
    \begin{equation}\label{eqnVarphiC2}
      -\lap \bar \varphi_c + A v \cdot \grad \bar \varphi_c
	  + \lambda \bar \varphi_c
	\geq
	  \alpha_1 \Chi*{\mc E_\eps^+}
	  +\abs{\grad \theta}^2 \paren[\big]{ - g'' + A b g' },
    \end{equation}
    where
    \begin{equation}\label{eqnB}
      b = b(\theta) = b_0\Chi{g' \geq 0} + b_1 \Chi{g' < 0}.
    \end{equation}
    This leads to an equation for $g$.
    Explicitly, we choose $g$ to be a periodic function so that 
    \begin{gather}
      \label{eqnG}
	- g'' + A b g' \geq f,
    \end{gather}
    normalized so that $\min g = 0$.
    
    Here the function $f$ is defined by
    \begin{equation*}
      f(\theta)
	\defeq
	\begin{dcases}
	  \frac{\alpha_2}{\sqrt{h_0^2 + \theta^2}} & \text{when } \abs{\theta} \leq \beta_0,\\
	  -\alpha_3 & \text{when } \beta_0 < \theta < \frac{\pi}{4} - \beta_0,
	\end{dcases}
    \end{equation*}
    in the interval $[-\beta_0, \pi/4 - \beta_0]$, and is extended periodically outside.

    The parameters $\alpha_2$ and $\alpha_3$ are chosen as follows.
    We require $\alpha_2$ to be large enough so that
    \begin{equation*}
      \abs{\grad \theta}^2 f
	= \frac{\alpha_2 \abs{\grad\theta}^2}{\sqrt{h_0^2 + \theta^2}}
	\geq 1
      \quad\text{in } \bar{\mc C}_\eps^+ \cap \{h \geq h_0\}.
    \end{equation*}
    Equation~\eqref{eqnGradTheta} guarantees that such a choice of $\alpha_2$ is possible.
    The role of the parameter $\alpha_3$ is to guarantee that the equation~\eqref{eqnG} admits a periodic super-solution.
    We show below that $\alpha_3 \approx \ln h_0$.
    
    Now equation~\eqref{eqnVarphiC2} shows that when $h \geq h_0$ we have
    \begin{equation*}
      -\lap \bar \varphi_c + A v \cdot \grad \bar \varphi_c
	  + \lambda \bar \varphi_c
	\geq
	  \alpha_1 \Chi*{\mc E_\eps}
	  +\abs{\grad \theta}^2 f
	\geq 
	  (\alpha_1 - \alpha_3) \Chi*{\mc E_\eps^+}
	  + \Chi*{\bar{\mc C}_\eps^+}
	\geq 1
    \end{equation*}
    provided $\alpha_1$ is chosen so that
    \begin{equation*}
      \alpha_1 \geq 1 + \alpha_3 \approx \abs{\ln \delta}.
    \end{equation*}
    Thus the maximum now principle guarantees
    \begin{equation*}
      \varphi_c \leq \psi(h_0) + \bar \varphi_c
      \quad \text{on all of } \mc B_\eps^+.
    \end{equation*}

    To finish the proof, we only need to estimate $\bar \varphi_c$.
    Tracing through the above we see 
    \begin{equation*}
      \varphi_c(x)
	\leq \psi(h_0) + \alpha_1  \varphi_e(x) +\norm{g}_{L^\infty}
	\leq c \abs{\ln \delta} \paren[\Big]{ \delta^2 + \frac{ h(x) }{\sqrt{\lambda}} } + \norm{g}_{L^\infty}. 
    \end{equation*}
    We claim
    \begin{equation}\label{eqnGBound}
      \sup_\theta g \leq c \delta^2 \abs{\ln \delta}
    \end{equation}
    from which~\eqref{eqnEtaMainBound} immediately follows.

    We remark that an elementary argument using Duhamel's formula quickly shows that
    \begin{equation*}
      \norm{g}_{L^\infty} \leq \frac{c \norm{f}_{L^1}}{A}.
    \end{equation*}
    If we apriori knew $\alpha_3 \approx \abs{\ln \delta}$, this would imply $\norm{f}_{L^1} \leq c \abs{\ln \delta}$ which immediately gives~\eqref{eqnGBound}.
    If, for instance, $b$ was constant (and not a nonlinear function of $g'$), then the solvability condition for~\eqref{eqnG} is precisely
    \begin{equation*}
      \int_{0}^{2\pi} f = 0,
    \end{equation*}
    from which we obtain
    \begin{equation*}
      \alpha_3 = \int_{-\beta_0}^{\beta_0} \frac{\alpha_2}{\sqrt{h_0^2 + \theta^2}} \, d\theta \leq c \abs{\ln \delta}.
    \end{equation*}
    Unfortunately the nonlinear dependence of $b$ on $g'$ doesn't allow this simple argument to work.
    We are instead forced to use a somewhat technical construction of a super-solution and explicitly show $\alpha_3 \approx \ln A$.
    We devote the rest of this proof to this construction and proving~\eqref{eqnGBound}.

    To find a super-solution to~\eqref{eqnG} we treat it as a first order ODE for $g'$.
    We will construct a piecewise $C^\infty$ function $G$ which is $2\pi$ periodic such that
    \begin{gather}\label{eqnGG}
      G' \leq A b G - f.
    \end{gather}
    Once $G$ has been constructed, define
    \begin{equation}\label{eqnGDef}
      g(\theta) = \int_{\theta_0}^\theta G(t) \, dt.
    \end{equation}
    where $\theta_0$ is chosen to ensure our normalization condition $\min g = 0$.
    In order for $g$ to be $2\pi$ periodic we need to ensure
    \begin{equation}\label{eqnGSolvability}
      \int_0^{2\pi} G = 0.
    \end{equation}
    Further, in order for $g$ to be a super-solution to~\eqref{eqnG} we need to ensure that all discontinuities of $G$ only have downward jumps.

    The main idea behind constructing $G$ is to solve the ODE~\eqref{eqnGG} \emph{backwards} with final condition $G(\frac{\pi}{2} + \beta_0) = 0$.
    Using~\eqref{eqnGG} one can apriori compute regions where $G > 0$, and so determine $b$ explicitly.
    Once this is known, $G$ can be found explicitly.

    To flesh out the details of this approach, let $\theta_1 = \frac{\pi}{2} + \beta_0$ and define
    \begin{equation}\label{eqnGG1}
      G(\theta)
	= G_1(\theta)
	\defeq
	  \int_{\theta}^{\theta_1} f(t) e^{-A b_0 (\theta_1 - t)} \, dt
      \quad\text{for }
	\frac{\pi}{2} - \beta_0
	\leq
	\theta \leq \theta_1 \defeq \frac{\pi}{2} + \beta_0.
    \end{equation}
    Since $f \geq 0$ in corners, the above definition forces $G_1 \geq 0$.
    Further, when $G \geq 0$, we know $b = b_0$.
    Consequently $G$ satisfies~\eqref{eqnGG} for $\theta \in (\frac{\pi}{2} - \beta_0, \frac{\pi}{2} + \beta_0)$.

    Note that~\eqref{eqnGG1} also implies $A b_0 G \to f$ uniformly on $(\frac{\pi}{2} - \beta_0, \frac{\pi}{2} + \beta_0)$ as $A \to \infty$.
    Thus
    \begin{equation*}
      0 < G_1(\frac{\pi}{2} - \beta_0) \leq c
      \quad\text{and}\quad
      \int_{\frac{\pi}{2} - \beta_0}^{\frac{\pi}{2} + \beta_0} G_1
      \leq \frac{c}{A} 
	\int_{\frac{\pi}{2} - \beta_0}^{\frac{\pi}{2} + \beta_0} f
      \leq \frac{c \ln A}{A}.
    \end{equation*}

    Now we extend $G$ to the left of $\frac{\pi}{2} - \beta_0$.
    Define
    \begin{equation*}
      G_2(\theta)
	= \frac{-\alpha_3}{A b_0}
	  + \paren[\Big]{
	      G_1\paren[\big]{ \frac{\pi}{2} - \beta_0}
	      + \frac{\alpha_3}{A b_0}
	    }
	    e^{-A (\frac{\pi}{2} - \beta_0 - \theta)},
      \quad\text{and}\quad
      \theta_2 = G_2\inv(0).
    \end{equation*}
    An explicit calculation shows that $\theta_2$ is uniquely defined and
    \begin{equation*}
      \theta_2 \geq \frac{\pi}{2} - \beta_0 - \frac{c \abs{\ln (\alpha_3 / A)}}{A}.
    \end{equation*}
    With this we define
    \begin{equation*}
      G(\theta) = G_2( \theta) \quad \text{for } \theta_2 \leq \theta \leq \frac{\pi}{2} - \beta_0.
    \end{equation*}
    Since $G \geq 0$ on this interval $b = b_0$ and so $G$ satisfies~\eqref{eqnGG} on $(\theta_2, \frac{\pi}{2} - \beta_0)$.

    Finally, we define
    \begin{equation*}
      G(\theta)
	= G_3(\theta) \defeq \frac{-\alpha_3}{A b_1}
	  \paren[\Big]{ 1 -  e^{-A (\theta_2 - \theta)} }
	\quad\text{for } \beta_0 \leq \theta \leq \theta_2,
    \end{equation*}
    and extend $G$ periodically outside the interval $(\beta_0, \frac{\pi}{2} + \beta_0)$.
    Observe $G_3 \leq 0$ on the interval $(\beta_0, \theta_2)$ and hence $b = b_1$ and $G$ satisfies~\eqref{eqnGG} on this interval as well.

    By construction note that $G$ is continuous (and piecewise smooth) on $[\beta_0, \frac{\pi}{2} + \beta_0]$.
    Further, since $G_3(\beta_0) < 0$ and $G_1(\frac{\pi}{2} + \beta_0) = 0$ the discontinuity introduced by extending $G$ periodically will only have a downward jump.
    Moreover as 
    \begin{equation*}
      \int_{\beta_0}^{\theta_2} G_3 \approx \frac{-\alpha_3}{A},
      \quad
      \int_{\theta_2}^{\frac{\pi}{2} - \beta_0} G_2 \approx  \frac{\abs{\ln (\alpha_3 / A)}}{A},
      \quad\text{and}\quad
      \int_{\frac{\pi}{2} - \beta_0}^{\frac{\pi}{2} + \beta_0} G_1
	\approx \frac{\ln A}{A},
    \end{equation*}
    there exists $\alpha_3 = O(\ln A)$ so that
    \begin{equation*}
      \int_{\beta_0}^{\frac{\pi}{2} + \beta_0} G
	=
	  \int_{\beta_0}^{\theta_2} G_3
	  +
	  \int_{\theta_2}^{\frac{\pi}{2} - \beta_0} G_2
	  +
	  \int_{\frac{\pi}{2} - \beta_0}^{\frac{\pi}{2} + \beta_0} G_1
	= 0.
    \end{equation*}
    By periodicity, this will imply~\eqref{eqnGSolvability}, completing the construction of $G$.

    It remains to prove~\eqref{eqnGBound}.
    For this note that our explicit construction of $G$ also shows
    \begin{equation*}
      \int_0^{2\pi} \abs{G} \leq \frac{c \ln A}{A}.
    \end{equation*}
    Since $g$ is defined by~\eqref{eqnGDef}, the desired bound~\eqref{eqnGBound} is immediate.
    This concludes the proof.
  \end{proof}

  \section{The \texorpdfstring{$n^\text{th}$}{nth} return to the separatrix.}\label{sxnNthReturn}

  The aim of this section we prove Lemma~\ref{lmaTauINBound}.
  The main step in the proof is Lemma~\ref{lmaTau1BoundLower} and was proved in Section~\ref{sxnFirstReturnMain}.
  To prove Lemma~\ref{lmaTauINBound} we first prove (Section~\ref{sxnCoordinateReturn}, Lemma~\ref{lmaTauI1Bound}) that Lemma~\ref{lmaTau1BoundLower} also holds for the stopping time $\tau_1^i$.
  Finally we use the strong Markov property and estimate $\tau_n^i$ in terms of $\tau_1^i$ (Section~\ref{sxnHigherReturn}).

  \subsection{The first return of the coordinate processes}\label{sxnCoordinateReturn}

  In this section we prove that Lemma~\ref{lmaTauINBound} holds for $n = 1$.
  For clarity, we state this result as a separate lemma below.

  \begin{lemma}\label{lmaTauI1Bound}
    There exists a positive constant $c_0$ such that if $i \in \{1, 2 \}$, $x \in \mc B_\delta$ and~\eqref{eqnTRange} holds, then
    \begin{gather}
      \label{eqnTau1iLowerExplicit}
      \inf_{ \abs{h(x)} < \delta } \prob^x( \tau^i_1 \leq t ) \geq
	\paren[\Big]{1-  \frac{c_0 \delta \abs{\ln \delta} }{\sqrt{t}} }^+,
      \\
      \llap{\text{and}\qquad}
      \label{eqnTau1iUpperExplicit}
	\sup_{ \abs{h(x)} < \delta } \prob^x( \tau^i_1 \leq t ) \leq
	  1-\erf  \paren[\Big]{ \frac{\delta }{c_0 \sqrt{t}} }.
    \end{gather}
  \end{lemma}
  \begin{proof}
    We begin with the lower bound~\eqref{eqnTau1iLowerExplicit}.
    For $i \in \{1, 2\}$ define
    \begin{equation*}
      p_i
	= \sup_{x \in \R^2} \prob^x(\tau_1^i = \tau_1)
	= \sup_{x \in \R^2} \prob^x(X_i\paren{\tau_1} \in \pi \Z).
    \end{equation*}
    By definition of $\tau_1$ we recall that $X$ has to first hit $\del B_\delta$ before returning to the separatrix.
    Consequently
    \begin{equation*}
      p_i
	= \sup_{x \in \mc B_\delta} \prob^x(X_i\paren{\tau} \in \pi \Z),
    \end{equation*}
    where $\tau = \inf\{ t \geq 0 \st h(X_t) = 0 \}$ is the first hitting time of $X$ to the separatrix $\{h = 0\}$.
    By symmetry of the flow, we may further restrict the supremum above to only run over $x \in Q_0$, where $Q_0 = (0, \pi)^2$.
    In $Q_0$ we know $\varphi_i(x) = \prob^x(X_i\paren{\tau} \in \pi \Z)$ satisfies the cell problem
    \begin{equation*}
      \begin{beqn}
	- \lap \varphi_i + A v \cdot \grad \varphi_i  = 0 &
	  in $Q_0$,\\
	\varphi_i = \Chi{x_i \in \set{0, \pi} } & on $\del Q_0$.
      \end{beqn}
    \end{equation*}
    From this it immediately follows~\cites{FreidlinWentzell12,NovikovPapanicolaouEtAl05} that $p_i$ is bounded away from both $0$ and $1$ by a constant that is independent of $A$.

    Now we will prove the lower bound~\eqref{eqnTau1iLowerExplicit} by first showing
    \begin{equation}\label{eqnFnMain}
      \prob^x(\tau_1^1 = \tau_n \AND \tau_n \geq  t )
	\leq  f_n(t)
	\defeq p_1 p_2^{n-1}
	  \min \paren[\Big]{ 1,  \frac{c_1 \delta n \abs{\ln \delta}}{\sqrt{t}} },
    \end{equation}
    for some constant $c_1$ independent of $n$ and $A$.
    Indeed, once~\eqref{eqnFnMain} is established, we see
    \begin{equation*}
      \prob^x( \tau_1^1 \geq t )
	= \sum_1^\infty \prob(\tau_1^1 = \tau_n \AND \tau_n \geq t )
	\leq \sum_1^\infty  p_1 p_2^{n-1} \frac{c_1 n \delta \abs{\ln \delta}}{\sqrt{t}}
	\leq \frac{c \delta \abs{\ln \delta}}{\sqrt{t}},
    \end{equation*}
    from which~\eqref{eqnTau1iLowerExplicit} immediately follows for $i = 1$.
    The case $i = 2$ follows by symmetry of the flow.
    As usual we assume that $c > 0$ is a finite constant, independent of $A$, that may increase from line to line.

    We prove~\eqref{eqnFnMain} by induction.
    For $n = 1$, Lemma~\ref{lmaTau1BoundLower} gives
    \begin{equation*}
      \prob^x(\tau_1^1 = \tau_1 \AND \tau_1 \geq  t )
	\leq \prob^x( \tau_1 \geq  t )
	\leq \frac{c_1' \delta \abs{\ln \delta}}{\sqrt{t}},
    \end{equation*}
    for some constant $c_1'$ that is independent of $A$.
    Since $p_1 > 0$ this yields~\eqref{eqnFnMain} for $n = 1$.

    For the inductive step, define $q_n = \prob^x( \tau_1^1 = \tau_n \AND \tau_1 \geq t )$.
    By the strong Markov property
    \begin{align}
      \nonumber
      \MoveEqLeft
      \prob^x( \tau_1^1 = \tau_n \AND \tau_n \geq t )\\
	\nonumber
	&= q_n
	  + \E^x \brak[\Big]{
	      \Chi{\tau_1 < t \AND X_1(\tau_1) \not\in \pi \Z}
	      \prob^{X(\tau_1) } \paren[\big]{
		\tau_{n-1} + s \geq t, \AND \tau_{n-1} = \tau^1_1
	      }
	      \at{s = \tau_{1}}
	    }
	\\
	\nonumber
	&\leq
	  q_n
	  + \E^x \brak[\Big]{
	      \Chi{\tau_1 = \tau^2_1 < t }
	      f_{n-1}( t - \tau_{1})
	    }
	\\
	\nonumber
	&= q_n - \int_0^t f_{n-1}( t - s ) \, d\prob^x( \tau_{1} \geq s  \AND \tau_1 = \tau^2_1 )  \\
	\nonumber
	&= q_n -
	    \brak[\Big]{
	      f_{n-1}(t-s) \prob^x(\tau_{1} \geq s  \AND \tau_1 = \tau^2_1 )
	    }_0^t \qquad
	    \\
	  \nonumber
	  &\qquad\qquad
	    -\int_0^t
	      \prob^x\paren{
		\tau_{1} \geq s  \AND 
		\tau_1 = \tau^2_1
	      }
	      f_{n-1}^{\prime}( t - s ) \, ds
	\\
	\label{eqnTau111}
	&= q_n - p_1 p_2^{n-2} \prob^x(\tau_{1} \geq t  \AND \tau_1 = \tau^2_1 ) \\
	  \nonumber
	  &\qquad\qquad
	    + f_{n-1}(t) p_2 
	    - \int_0^t \prob^x( \tau_{1} \geq s  \AND 
	    \tau_1 = \tau^2_1 ) f_{n-1}^{\prime}( t - s ) \, ds.
    \end{align}

    We claim that the net contribution of the first two terms in~\eqref{eqnTau111} is negative.
    Indeed, using the strong Markov property again gives
    \begin{multline}\label{eqnQnMain}
      q_n
	= \prob^x( \tau_1^1 = \tau_n \AND \tau_1 \geq t )
	= \E^x
	    \Chi{ \tau_1 \geq t \AND X_1(\tau_1) \not\in \pi \Z }
	    \prob^{X(\tau_1)} \paren[\big]{ \tau_1^1 = \tau_{n-1} }
      \\
	\leq
	  \prob^x \paren{ \tau_1 \geq t \AND \tau_1 = \tau_1^2}
	  \sup_{h(y) = 0}\prob^{y} \paren[\Big]{ \tau_1^1 = \tau_{n-1} }
	\leq
	  p_1 p_2^{n-2}
	  \prob^x \paren{ \tau_1 \geq t }.
    \end{multline}
    The last inequality above followed by repeated use of the strong Markov property.

    Now using~\eqref{eqnQnMain} and the inductive hypothesis in~\eqref{eqnTau111} yields
    \begin{equation}\label{eqnTau112}
      \prob^x( \tau_1^1 = \tau_n \AND \tau_n \geq t )
	\leq p_1 p_2^{n-1}
	  \Big[
	  \begin{multlined}[t]
	    \frac{c_1 \delta (n-1)  \abs{\ln \delta}}{\sqrt{t}} \mathop+
	    \\
	      \int_0^{t_n}
		\paren[\Big]{\frac{c_1' \delta  \abs{\ln \delta}}{p_1 \sqrt{s}} }
		\paren[\Big]{\frac{c_1 \delta  (n-1) \abs{\ln \delta}}{2 (t-s)^{3/2}} }
		\, ds
	    \Big],
	  \end{multlined}
    \end{equation}
    where $t_n < t$ is chosen so that
    \begin{equation*}
      \frac{c_1 (n-1)  \delta \abs{\ln \delta}}{\sqrt{t - t_n} } = 1.
    \end{equation*}

    The integral on the right of~\eqref{eqnTau112} can be computed explicitly from the identity
    \begin{equation}\label{eqnConvId}
      \int \frac{ds}{2 \sqrt{ s (t-s)^3 } } = \frac{\sqrt{s}}{t \sqrt{t - s}}.
    \end{equation}
    Consequently,
    \begin{equation*}
      \prob^x( \tau_1^1 = \tau_n \AND \tau_n \geq t )
      \leq
	\frac{p_1 p_2^{n-1} \delta \abs{\ln \delta}}{\sqrt{t}}
	\brak[\Big]{(n-1) c_1 + \frac{c_1'}{p_1}}.
    \end{equation*}
    Note that $c_1' / p_1$ is independent of $n$, and so choosing $c_1 > c_1' / p_1$ gives
    \begin{equation}\label{eqnTau113}
      \prob^x( \tau_1^1 = \tau_n \AND \tau_n \geq t ) \leq
	\frac{p_1 p_2^{n-1} n c_1 \delta \abs{\ln \delta}}{\sqrt{t}}.
    \end{equation}
    Further, by repeated use of the strong Markov property we see
    \begin{equation}\label{eqnTau114}
      \prob^x( \tau_1^1 = \tau_n \AND \tau_n \geq t ) \leq
	p_1 p_2^{n-1}.
    \end{equation}
    Combining~\eqref{eqnTau113} and~\eqref{eqnTau114} immediately gives~\eqref{eqnFnMain}.
    As explained earlier this implies~\eqref{eqnTau1iLowerExplicit} and completes the proof of the lower bound.
    \medskip

    Next, we turn to the upper bound~\eqref{eqnTau1iUpperExplicit}.
    Clearly
    \begin{equation*}
      \prob^x\paren{ \tau_i^1 \leq t }
	\leq \prob^x\paren{ \tau_1 \leq t }
	\leq \prob^x\paren[\big]{ \tau(X_{\sigma_1}) \leq t }
	\leq \sup_{\abs{h(y)} = \delta} \prob^y( \tau \leq t ).
    \end{equation*}
    Now we let $\psi(y, t) = \prob^y( \tau \leq t )$, and restrict our attention to one cell $Q$.
    We know that $\psi$ satisfies
    \begin{equation}\label{eqnPDEpsi}
      \begin{beqn}
	\delt \psi + A v \cdot \grad \psi - \lap \psi = 0 & in $Q$,\\
	\psi = 0 & on $Q \times \{0\}$,\\
	\psi = 1 & on $\del Q \times (0, \infty)$.
      \end{beqn}
    \end{equation}

    We claim
    \begin{equation*}
      \psi^+ \defeq 1 - \erf \paren[\Big]{ \frac{h}{\sqrt{c_0 t}} }
    \end{equation*}
    is a super-solution to~\eqref{eqnPDEpsi}.
    Indeed, $\psi^+$ certainly satisfies the boundary and initial conditions in~\eqref{eqnPDEpsi}.
    Further,
    \begin{equation}\label{eqnPsiPlus}
      \delt \psi^+ + A v \cdot \grad \psi^+ - \lap \psi^+
	=
	  \paren[\Big]{
	    \frac{h}{c_0 \sqrt{\pi t^3} }
	    - \frac{4 h}{c_0 \sqrt{\pi t}}
	    - \abs{\nabla h}^2  \frac{4 h}{c_0^3 \sqrt{\pi t^3}}
	  }
	  \exp \paren[\Big]{ - \frac{h^2}{c_0^2 t} }.
    \end{equation}
    Choosing $c_0$ large enough we can ensure
    \begin{equation*}
      1 - 4t - \frac{4 \abs{\grad h}^2}{ c_0^2 } \geq 0
    \end{equation*}
    for $t \leq 1/8$.%
    This forces the right hand side of~\eqref{eqnPsiPlus} to be positive, showing $\psi^+$ is a super-solution to~\eqref{eqnPDEpsi}.

    Consequently if $\abs{h(x)} < \delta$ we must have $\prob^x( \tau_1 \leq t ) \leq  \psi^+(\delta)$ which immediately implies~\eqref{eqnTau1iUpperExplicit}.
  \end{proof}

  \subsection{Higher return times of the coordinate processes.}\label{sxnHigherReturn}
  The next step is to estimate $\tau_n^i$ in terms of $\tau_1^i$.
  This follows abstractly from the strong Markov property, and is our next lemma.
  \begin{lemma}\label{lmaConvolutionBound}
    Fix $T \in (0, \infty]$ and $i \in \set{1,2}$.
    Suppose $f$ and $g$ are two absolutely continuous, increasing functions such that $f(0)=g(0) = 0$ and
    \begin{gather}
      \label{eqnTau1BoundLower}
	\inf_{\abs{h(x)} < \delta} \prob^x( \tau_1^i \leq t )
	\geq f(t)
      \\
      \label{eqnTau1BoundUpper}
	\llap{\text{and}\qquad}
	\sup_{\abs{h(x)} < \delta} \prob^x( \tau_1^i \leq t )
	\leq g(t)
    \end{gather}
    for all $t \leq T$.
    Then, for any $n \in \N$ and $t \leq T$ we have
    \begin{gather}
      \label{eqnTauNBound}
      \inf_{ \abs{h(x)} < \delta } \prob^x( \tau_n^i \leq t )
	 \geq  f * (f')^{*(n-1)}
	 = \int_0^t (f')^{*n},\\
       \llap{\text{and}\qquad}
       \label{eqnTauNBound2}
       \sup_{ \abs{h(x)} < \delta } \prob^x( \tau_n^i \leq t )
	 \leq   g * (g')^{*(n-1)}
	 = \int_0^t (g')^{*n}.
    \end{gather}
    The convolutions above are defined using
    \begin{equation*}
      f_1*f_2(t) = \int_0^t f_1(s) f_2(t-s) \, ds,
      \quad\text{and}\quad
      f_1^{*n} = \underbrace{f_1 * \cdots * f_1}_{n \text{ times}}.
    \end{equation*}
  \end{lemma}
  \begin{proof}[Proof of  Lemma~\ref{lmaConvolutionBound}]
    By induction
    \begin{multline*}
      \prob^x( \tau_n \leq t )
	= \E^x \left[\prob^{X_{\tau_{n-1}} } (\tau_1 + s \leq t)\right]_{s = \tau_{n-1}}\\
	\geq \E^x f( t - \tau_{n-1} )
	= \int_0^t f( t - s ) \, dP^x( \tau_{n-1} \leq s )\\
	= \bigl[f(t-s) \prob(\tau_{n-1} \leq s)\bigr]_0^t + \int_0^t \prob^x( \tau_{n-1} \leq s ) f'( t - s ) \, ds\\
	\geq \int_0^t  f * (f')^{*(n-2)}(s) f'( t - s ) \, ds
	= f * (f')^{*(n-1)}(t)
	 = \int_0^t (f')^{*(n)}(s) \, ds
    \end{multline*}
    proving~\eqref{eqnTauNBound} as desired.
    The proof of~\eqref{eqnTauNBound2} is identical.
  \end{proof}

  Finally, we conclude this section by using Lemmas~\ref{lmaTauI1Bound} and~\ref{lmaConvolutionBound} to prove the bounds claimed in Lemma~\ref{lmaTauINBound}.
  \begin{proof}[Proof of Lemma~\ref{lmaTauINBound}]
    To prove the lower bound~\eqref{eqnTauINLower}, define the function $f_a$ by
    \begin{equation*}
      f_a(t) = \Chi{t > 0} \paren[\Big]{1 - \frac{a}{\sqrt{t}} }^+.
    \end{equation*}
    Then for $a, b > 0$, a direct calculation using~\eqref{eqnConvId} shows that for $t \geq a^2 + b^2$ we have
    \begin{equation*}
      f_b * f_a'(t)
	= \int_{a^2}^{t - b^2} \frac{a}{2 s^{3/2}} \paren[\Big]{ 1 - \frac{b}{\sqrt{t - s}} } \, ds
	= 1 - \frac{a \sqrt{t - b^2} + b \sqrt{t - a^2} }{t}
	\geq 1 - \frac{a + b}{\sqrt{t}}
    \end{equation*}
    Consequently
    \begin{equation}\label{eqnFaConvFb}
      f_b * f_a'(t) \geq f_{a+b}(t).
    \end{equation}
    Now~\eqref{eqnTauINLower} immediately follows from~\eqref{eqnTau1iLowerExplicit}, \eqref{eqnTauNBound} and~\eqref{eqnFaConvFb}.

    Inequality~\eqref{eqnTauINUpper} follows using an exact calculation for $(g')^{*n}$ using the Laplace transform.
    Namely, let $c_0$ be the constant in Lemma~\ref{lmaTauI1Bound} and observe
    \begin{equation*}
      \mathcal L g'(s)
	\defeq \int_0^\infty e^{-st} g'(t) \, dt
	  = \frac{\delta }{c_0 \sqrt{\pi}}
	    \int_0^\infty e^{-st}
	      \exp\paren[\Big]{ \frac{- \delta^2}{c_0^2 t} }
	      \, \frac{dt}{t^{3/2}}
	  = \exp\paren[\Big]{ \frac{- 2 \delta \sqrt{s}}{c_0} }.
    \end{equation*}
    Consequently
    \begin{equation*}
      \mathcal L (g')^{*n}(s)
	= \exp\paren[\Big]{\frac{- 2 n \delta  \sqrt{s}}{c_0} }
	= \mathcal L (g')( n^2 s ),
    \end{equation*}
    and so
    \begin{equation*}
      (g')^{*n} (t)
	= \frac{1}{n^2} g'\paren[\big]{ \frac{t}{n^2} }
	= \frac{n \delta }{c_0 \sqrt{\pi} t^{3/2}}
	  \exp\paren[\Big]{ \frac{- n^2 \delta^2 }{c_0^2 t} }.
    \end{equation*}
    Integrating in time and using~\eqref{eqnTau1iUpperExplicit} and \eqref{eqnTauNBound2}  we obtain~\eqref{eqnTauINUpper} as desired.
  \end{proof}

  \section{The variance bound}\label{sxnVarBound}
  Finally, we conclude this paper with a proof of Lemma~\ref{lmaVarNBounds}.
  Our proof is similar in spirit to Lemma~\ref{lmaBoundaryExit}, and relies on the fact that the chance that $X$ re-enters a cell corner is bounded away from~$1$ (Lemma~\ref{lmaEdgeToCorner}).

  We prove the upper and lower bounds separately.
  \begin{proof}[Proof of the upper bound in Lemma~\ref{lmaVarNBounds}]
    Without loss of generality we assume $i = 1$.
    By the strong Markov property
    \begin{multline}\label{eqnVarNConditioned}
      \E^x\paren[\big]{
	\abs{ X_{1}(\tau^1_{n+1} \varmin t ) -X_1(\tau^1_n \varmin t) }^2 
	\big|
	\mathcal{F}_{\tau^1_n} }\\
      = \Chi{ \tau^1_n \leq t}  
	  \E^{X_{\tau_n^1}} \abs{ X_{1}(\tau^1_{1} \varmin s) - X_1(0)}^2 
      \Bigr|_{s=t- \tau^1_n}.
    \end{multline}
    Hence
    \begin{equation}\label{eqnVarX1tmp1}
      \E^x \abs{ X_{1}(\tau^1_{n+1} \varmin t ) -X_1(\tau^1_n \varmin t) }^2 
	\leq \prob^x(\tau^1_n \leq t) \tilde V(t),
    \end{equation}
    where
    \begin{equation*}
      \tilde V(t)
	= \sup_{y \in \mc B_\delta, ~ s \leq t}
	    \E^y \abs{ X_1(\tau^1_1 \varmin s) - y_1}^2.
    \end{equation*}
    Thus the proof of the upper bound~\eqref{eqnVarUpper} will follow if we can find an upper bound for $\tilde V$ that is independent of $A$.

    For convenience, we will first estimate $V$, where
    \begin{equation*}
      V(t)
	\defeq \sup_{y \in \mc C, s \leq t}
	    \E^y \abs{ X_1(\sigma^1_1 \varmin s) - y_1}^2.
    \end{equation*}
    If $Q \subset \R^2$ is a cell, we claim
    \begin{equation}\label{eqnVBound}
      V(t) \leq \frac{ \diam(Q)^2 }{1 - P_0}
      \quad\text{and}\quad
      \tilde V(t) \leq \frac{5}{1 - P_0} \diam(Q)^2.
    \end{equation}
    Here $P_0$ is the constant appearing in Lemma~\ref{lmaEdgeToCorner}.
    Once~\eqref{eqnVBound} is established, combining it with~\eqref{eqnVarX1tmp1} immediately yields~\eqref{eqnVarUpper} as desired.

    To prove~\eqref{eqnVBound}, fix $\mc C_0$ to be a connected component of $\mc C$ (defined in~\eqref{eqnCornerEdgeDef}), and suppose $y \in \overline{\mc C}_0 \cap \mc B_\delta$.
    Define $\tilde \tau = \inf\{ t \geq 0 \st X_t \in \overline{\mc C} - \overline{\mc C}_0 \}$ be the hitting time of $X$ to a different corner.
    Clearly
    \begin{multline}\label{eqnV1}
      \E^y \abs{X_1(\sigma^1_1 \varmin t) - y_1}^2
	=\\
      \E^y \Chi{\tilde \tau < \sigma_1 \varmin t}
	  \abs{X_1(\sigma^1_1 \varmin t) - y_1}^2
	+ \E^y \Chi{\tilde \tau \geq \sigma_1 \varmin t}
	  \abs{X_1(\sigma^1_1 \varmin t) - y_1}^2,
    \end{multline}
    and we handle each term on the right individually.

    When $\sigma_1 \varmin t \leq \tilde \tau$, the process $X$ couldn't have travelled further than one side of the cell $Q$,
    and hence
    \begin{equation*}
      \E^y \Chi{\tilde \tau \geq \sigma_1 \varmin t}
	    \abs{X_1(\sigma^1_1 \varmin t) - y_1}^2
	\leq \diam(Q)^2
    \end{equation*}

    For the other term on the right of~\eqref{eqnV1} observe
    \begin{multline*}
      \E^y \Chi{\tilde \tau < \sigma_1 \varmin t}
	    \abs{X_1(\sigma^1_1 \varmin t) - y_1}^2
	= \E^y \Chi{\tilde \tau < \sigma_1 \varmin t} \E^{X_{\tilde \tau}}
	  \brak[\Big]{ \abs{X_1(\sigma_1 \varmin s) - y_1 }^2 }_{s = t - \tilde \tau}\\
	\leq \E^y \Chi{\tilde \tau < \sigma_1 \varmin t} V(t)
	\leq \prob^y \paren{\tilde \tau < \sigma_1} V(t)
	\leq P_0 V(t).
    \end{multline*}
    Note that the last inequality above follows immediately from Lemma~\ref{lmaEdgeToCorner}.
    Indeed, for $X$ to enter another corner before exiting the boundary layer, it must first enter an edge.
    From an edge (more precisely, from $\mc E'$), Lemma~\ref{lmaEdgeToCorner} shows that the chance that $X$ enters a corner is bounded above by $P_0 < 1$, and is independent of $A$.

    Combining our estimates and returning to~\eqref{eqnV1} we see
    \begin{equation*}
      \E^y \abs{X_1(\sigma^1_1 \varmin t) - y_1}^2
	\leq \diam(Q)^2 + P_0 V(t).
    \end{equation*}
    Taking the supremum over $y \in \mc C$, and using the fact that $V$ is increasing gives
    \begin{equation*}
      V(t) \leq \diam(Q)^2 + P_0 V(t),
    \end{equation*}
    from which the first inequality in~\eqref{eqnVBound} follows.

    Finally, we prove the second inequality in~\eqref{eqnVBound}.
    Observe
    \begin{equation*}
      \abs{X_1( \tau_1^1 \varmin t ) - X_1( \sigma_1^1 \varmin t )} \leq \diam(Q),
    \end{equation*}
    and hence
    \begin{equation}\label{eqnVCorner1}
      \E^y \abs{ X_1(\tau^1_1 \varmin t) - y_1}^2
	\leq
	  2 \paren[\big]{ \frac{1}{1 - P_0} + 1 } \diam(Q)^2
	\leq  \frac{4}{1 - P_0} \diam(Q)^2,
    \end{equation}
    for all $y \in \overline{\mc C}$.
    
    If $y \not\in \mc C$, then let $\tilde \sigma = \inf\{ t > 0 \st X_t \in \overline{\mc C}\}$ be the hitting time to the corner.
    Note
    \begin{align*}
      \MoveEqLeft
      \E^y \abs{ X_1(\tau^1_1 \varmin t) - y_1}^2
      \\
	&= \E^y \Chi{\tilde \sigma < \tau^1_1 \varmin t} \abs{ X_1(\tau^1_1 \varmin t) - y_1}^2
	+ \E^y \Chi{\tilde \sigma \geq \tau^1_1 \varmin t} \abs{ X_1(\tau^1_1 \varmin t) - y_1}^2\\
	&\leq
	  \sup_{
	    \substack{
	      s \leq t\\
	      z \in \del \mc C \cap \mc B_\delta
	    }
	  }
	  \E^z \abs{X_1(\tau^1_1 \varmin s ) - y_1}^2
	  + \diam(Q)^2
	\leq \frac{5}{1 - P_0} \diam(Q)^2.
    \end{align*}
    The first inequality above followed from the strong Markov property, and the last inequality above followed from~\eqref{eqnVCorner1}.
    This proves the second inequality in~\eqref{eqnVBound}, and finishes the proof of the upper bound in Lemma~\ref{lmaVarNBounds}.
  \end{proof}
  \begin{proof}[Proof of the lower bound in Lemma~\ref{lmaVarNBounds}]
    Using~\eqref{eqnVarNConditioned} we see that
    \begin{align*}
      \MoveEqLeft
      \E^x \abs{ X_{1}(\tau^1_{n+1} \wedge t ) -X_1(\tau^1_n \wedge t) }^2
      \\
      & \geq
	\E^x \paren[\big]{
	  \Chi{ \tau^1_n \leq t}
	  \inf_{\set{y \st[] y_1 \in \pi \Z} }
	    \E^y \abs{ X_{1}(\tau^1_{1} \varmin s) - y_1}^2
	      \Bigr|_{s=t- \tau^1_n}
	  }
	\\
      & = \int_0^t 
	  \inf_{\set{y \st[] y_1 \in \pi \Z} }
	    \E^y \abs{ X_{1}(\tau^1_{1} \varmin (t-s) ) - y_1}^2
	  \, d\prob^x ( \tau^1_n \leq s )
	\\
      & \geq
	\prob^x \paren[\big]{ \tau^1_n \leq \frac{t}{2} } V_1(t),
    \end{align*}
    where
    \begin{equation}\label{eqnV1def}
      V_1(t) \defeq
	\inf_{
	  \substack{
	    t/2\leq s \leq t\\
	    y \in \pi \Z \times \R
	    }
	  }
	  \E^y\abs{ X^y_{1}(\tau^1_{1} \varmin s) -y_1}^2.
    \end{equation}

    Thus the lower bound~\eqref{eqnVarLower} will follow provided we show that for all~$t$ satisfying~\eqref{eqnTRange} we have
    \begin{equation}\label{eqnV1lower}
      V_1(t) \geq c.
    \end{equation}
    We devote the rest of the proof to establishing~\eqref{eqnV1lower}.

    By symmetry the infimum in~\eqref{eqnV1def} can be taken over only the set $\set{ y \st y_1 = 0 }$.
    Consequently,
    \begin{align*}
      V_1(t)
	&\geq
	  \inf_{ \set{ y \st[] y_1 = 0} }
	  \E^y {
	    \Chi{ \tau^1_1 \leq t/2} \abs{ X_{1}(\tau^1_{1}) }^2 }
	\\
      &\geq
	\pi^2
	  \inf_{ \set{ y \st[] y_1 = 0} }
	  \prob^y \paren[\big]{
	    \tau^1_1 \leq \frac{t}{2} \AND \abs{ X_{1}(\tau^1_{1}) } \geq \pi
	  }.
    \end{align*}
    To estimate the right hand side, observe
    \begin{multline}\label{eqnV11}
      \prob^{y}\paren[\big]{
	\tau^1_1 \leq \frac{t}{2} \AND \abs{ X_{1}(\tau^1_{1})} \geq \pi
      }
      \geq
	\prob^{y} \paren[\big]{\tau^1_1 \leq \frac{t}{2}}
	-\prob^y \paren[\big]{ \abs{ X_{1}(\tau^1_{1})} = 0 }
      \\
      \geq
	1 - \frac{c_0 \delta {\abs{\ln \delta}} }{\sqrt{t}}
	  - \prob^{y}\paren[\big]{ \abs{ X_{1}(\tau^1_{1})} = 0 }
      =
	\prob^{y}\paren[\big]{ \abs{ X_{1}(\tau^1_{1})} \geq \pi } 
	 - \frac{c_0 \delta {\abs{\ln \delta}} }{\sqrt{t}}
    \end{multline}
    where the second inequality followed from Lemma~\ref{lmaTauINBound}.

    By the strong Markov property,
    \begin{align}
      \nonumber
      \MoveEqLeft
      \prob^y\paren{ \abs{X_{1}(\tau^1_{1})} \geq \pi }
	= \E^y \prob^{X(\sigma_1^1)}\paren{ \abs{X_1(\tau^1_1)} \geq \pi }
	\\
      \nonumber
	&\geq \prob^y \paren{ \abs{X_1(\sigma_1^1)} < \pi }
	  \inf \set[\big]{
	    \prob^z \paren{ |X_1(\tau_1^1)| = \pi }
	    \st \abs{h(z)} = \delta \AND \abs{z_1} < \pi
	    }
	\\
      \label{eqnX11}
	&\geq \prob^y \paren{ \abs{X_1(\sigma_1^1)} < \pi }
	  \inf \set[\big]{
	    \prob^z \paren{ |X_1(\tau_1)| = \pi }
	    \st \abs{h(z)} = \delta \AND \abs{z_1} < \pi
	    }.
    \end{align}
    We bound each term on the right individually.

    When $\abs{h(z)} = \delta$, the stopping time $\tau_1$ is simply the first hitting time of $X$ to the separatrix $\set{h = 0}$.
    Consequently,
    \begin{equation*}
      \inf_{\substack{ \abs{h(z)} = \delta \\ \abs{z_1} < \pi }}
	\prob^z \paren{ \abs{X_1(\tau_1)} = \pi }
      = \inf_{\substack{ h(z) = \delta \\ z \in Q_0 }}
	\zeta(z),
    \end{equation*}
    where  $Q_0 \defeq (0, \pi )\times( 0, \pi)$ and $\zeta$ is a solution to the cell problem
    \begin{equation*}
    \begin{beqn}
      A v \cdot \grad \zeta - \lap \zeta = 0 &
	in $Q_0$,\\
      \zeta(x) = \Chi{x_1=\pi}(x) & on $\del Q_0$.
    \end{beqn}
    \end{equation*}
    We know (see for instance~\cite{NovikovPapanicolaouEtAl05,FreidlinWentzell12,RhinesYoung83}) that
    \[
      \inf_{h(y)=\delta } \zeta(y) \geq c,
    \]
    for some constant $c = c(N)$ independent of $A$.

    For the first term on the right of~\eqref{eqnX11}, we use Lemma~\ref{lmaEdgeToCorner} again.
    Suppose $y \in \R^2$ and $y_1 = 0$.
    If $\abs{X_1( \sigma_1^1 )} \geq \pi$, then the process $X$ must have travelled  through at least one edge and re-entered a corner before exiting the boundary layer.
    By Lemma~\ref{lmaEdgeToCorner}, this happens with probability at most $P_0 < 1$.
    Consequently,
    \begin{equation*}
      \prob^y \paren{ \abs{X(\sigma_1^1)} < \pi }
      = 1 - \prob^y \paren{ \abs{X(\sigma_1^1)} \geq \pi }
      \geq 1 - P_0.
    \end{equation*}

    Thus returning to~\eqref{eqnX11} we see
    \begin{equation*}
      \prob^y( X_1( \tau_1^1 \geq \pi) ) \geq c,
    \end{equation*}
    for some constant $c$ independent of $A$.
    Using this in~\eqref{eqnV11} we obtain~\eqref{eqnV1lower}, provided $\delta \abs{\ln \delta} \ll \sqrt t$.
    This completes the proof of the upper bound~\eqref{eqnVarUpper}.
  \end{proof}
  
  \bibliographystyle{habbrv}
  \bibliography{refs}

\begin{thebibliography}{10}

\bibitem{BensoussanLionsEtAl78}
A.~Bensoussan, J.-L. Lions, and G.~Papanicolaou.
\newblock {\em Asymptotic analysis for periodic structures}, volume~5 of {\em
  Studies in Mathematics and its Applications}.
\newblock North-Holland Publishing Co., Amsterdam-New York, 1978.

\bibitem{CardosoTabeling88}
O.~Cardoso and P.~Tabeling.
\newblock Anomalous diffusion in a linear array of vortices.
\newblock {\em EPL}, 7(3):225, 1988.

\bibitem{Childress79}
S.~Childress.
\newblock Alpha-effect in flux ropes and sheets.
\newblock {\em Phys.\ Earth Planet Inter.}, 20:172--180, 1979.

\bibitem{ChildressSoward89}
S.~Childress and A.~M. Soward.
\newblock Scalar transport and alpha-effect for a family of cat's-eye flows.
\newblock {\em J. Fluid Mech.}, 205:99--133, 1989.

\bibitem{DolgopyatFreidlinEtAl12}
D.~Dolgopyat, M.~Freidlin, and L.~Koralov.
\newblock Deterministic and stochastic perturbations of area preserving flows
  on a two-dimensional torus.
\newblock {\em Ergodic Theory Dynam. Systems}, 32(3):899--918, 2012.

\bibitem{DolgopyatKoralov08}
D.~Dolgopyat and L.~Koralov.
\newblock Averaging of {H}amiltonian flows with an ergodic component.
\newblock {\em Ann. Probab.}, 36(6):1999--2049, 2008.

\bibitem{DolgopyatKoralov13}
D.~Dolgopyat and L.~Koralov.
\newblock Averaging of incompressible flows on two-dimensional surfaces.
\newblock {\em J. Amer. Math. Soc.}, 26(2):427--449, 2013.

\bibitem{Fannjiang02}
A.~Fannjiang.
\newblock Time scales in homogenization of periodic flows with vanishing
  molecular diffusion.
\newblock {\em J. Differential Equations}, 179(2):433--455, 2002.

\bibitem{FannjiangPapanicolaou94}
A.~Fannjiang and G.~Papanicolaou.
\newblock Convection enhanced diffusion for periodic flows.
\newblock {\em SIAM J. Appl. Math.}, 54(2):333--408, 1994.

\bibitem{Fredlin64}
M.~I. Fre{\u\i}dlin.
\newblock The {D}irichlet problem for an equation with periodic coefficients
  depending on a small parameter.
\newblock {\em Teor. Verojatnost. i Primenen.}, 9:133--139, 1964.

\bibitem{FreidlinWentzell12}
M.~I. Freidlin and A.~D. Wentzell.
\newblock {\em Random perturbations of dynamical systems}, volume 260 of {\em
  Grundlehren der Mathematischen Wissenschaften [Fundamental Principles of
  Mathematical Sciences]}.
\newblock Springer, Heidelberg, third edition, 2012.
\newblock Translated from the 1979 Russian original by Joseph Sz{\"u}cs.

\bibitem{GuyonPomeauEtAl87}
E.~Guyon, Y.~Pomeau, J.~P. Hulin, and C.~Baudet.
\newblock Dispersion in the presence of recirculation zones.
\newblock {\em Nuclear Physics B - Proceedings Supplements}, 2:271 -- 280,
  1987.

\bibitem{HairerKoralovEtAl14}
M.~Hairer, L.~Koralov, and Z.~Pajor-Gyulai.
\newblock From averaging to homogenization in cellular flows - an exact
  description of the phase transition, 2014.
\newblock preprint.

\bibitem{HaynesVanneste14}
P.~H. Haynes and J.~Vanneste.
\newblock Dispersion in the large-deviation regime. {P}art 1: shear flows and
  periodic flows.
\newblock {\em J. Fluid Mech.}, 745:321--350, 2014.

\bibitem{HaynesVanneste14b}
P.~H. Haynes and J.~Vanneste.
\newblock Dispersion in the large-deviation regime. {P}art 2. {C}ellular flow
  at large {P}\'eclet number.
\newblock {\em J. Fluid Mech.}, 745:351--377, 2014.

\bibitem{Heinze03}
S.~Heinze.
\newblock Diffusion-advection in cellular flows with large {P}eclet numbers.
\newblock {\em Arch. Ration. Mech. Anal.}, 168(4):329--342, 2003.

\bibitem{KaloshinDolgopyatEtAl05}
V.~Kaloshin, D.~Dolgopyat, and L.~Koralov.
\newblock Long time behaviour of periodic stochastic flows.
\newblock In {\em X{IV}th {I}nternational {C}ongress on {M}athematical
  {P}hysics}, pages 290--295. World Sci. Publ., Hackensack, NJ, 2005.

\bibitem{KaratzasShreve91}
I.~Karatzas and S.~E. Shreve.
\newblock {\em Brownian motion and stochastic calculus}, volume 113 of {\em
  Graduate Texts in Mathematics}.
\newblock Springer-Verlag, New York, second edition, 1991.

\bibitem{Koralov04}
L.~Koralov.
\newblock Random perturbations of 2-dimensional {H}amiltonian flows.
\newblock {\em Probab. Theory Related Fields}, 129(1):37--62, 2004.

\bibitem{LiuXinEtAl11}
Y.-Y. Liu, J.~Xin, and Y.~Yu.
\newblock Asymptotics for turbulent flame speeds of the viscous {G}-equation
  enhanced by cellular and shear flows.
\newblock {\em Arch. Ration. Mech. Anal.}, 202(2):461--492, 2011.

\bibitem{NolenXinEtAl09}
J.~Nolen, J.~Xin, and Y.~Yu.
\newblock Bounds on front speeds for inviscid and viscous {$G$}-equations.
\newblock {\em Methods Appl. Anal.}, 16(4):507--520, 2009.

\bibitem{NovikovPapanicolaouEtAl05}
A.~Novikov, G.~Papanicolaou, and L.~Ryzhik.
\newblock Boundary layers for cellular flows at high {P}\'eclet numbers.
\newblock {\em Comm. Pure Appl. Math.}, 58(7):867--922, 2005.

\bibitem{Olla94}
S.~Olla.
\newblock {\em Lectures on Homogenization of Diffusion Processes in Random
  Fields}.
\newblock Publications de l'Ecole Doctorale de l'Ecole Polytechnique, 1994.

\bibitem{PavliotisStuart08}
G.~A. Pavliotis and A.~M. Stuart.
\newblock {\em Multiscale methods}, volume~53 of {\em Texts in Applied
  Mathematics}.
\newblock Springer, New York, 2008.

\bibitem{RhinesYoung83}
P.~B. Rhines and W.~R. Young.
\newblock How rapidly is passive scalar mixed within closed streamlines?
\newblock {\em J.\ Fluid Mech.}, 133:135--145, 1983.

\bibitem{RosenbluthBerkEtAl87}
M.~N. Rosenbluth, H.~L. Berk, I.~Doxas, and W.~Horton.
\newblock Effective diffusion in laminar convective flows.
\newblock {\em Phys.\ Fluids}, 30:2636--2647, 1987.

\bibitem{SaguesHorsthemke86}
F.~Sagues and W.~Horsthemke.
\newblock Diffusive transport in spatially periodic hydrodynamic flows.
\newblock {\em Physical Review A}, 34(5), 1986.

\bibitem{ShawThiffeaultEtAl07}
T.~A. Shaw, J.-L. Thiffeault, and C.~R. Doering.
\newblock Stirring up trouble: multi-scale mixing measures for steady scalar
  sources.
\newblock {\em Phys. D}, 231(2):143--164, 2007.

\bibitem{Shraiman87}
B.~Shraiman.
\newblock Diffusive transport in a raleigh-bernard convection cell.
\newblock {\em Phys.\ Rev.\ A}, 36:261--267, 1987.

\bibitem{Sowers06}
R.~B. Sowers.
\newblock Random perturbations of two-dimensional pseudoperiodic flows.
\newblock {\em Illinois J. Math.}, 50(1-4):853--959 (electronic), 2006.

\bibitem{Taylor53}
G.~Taylor.
\newblock Dispersion of soluble matter in solvent flowing slowly through a
  tube.
\newblock {\em Proc. R. Soc. Lond. A}, 219(1137):186--203, 1953.

\bibitem{ThiffeaultChildress10}
J.-L. Thiffeault and S.~Childress.
\newblock Stirring by swimming bodies.
\newblock {\em Phys. Lett. A}, 374(34):3487--3490, 2010.

\bibitem{YoungPumirEtAl89}
W.~Young, A.~Pumir, and Y.~Pomeau.
\newblock Anomalous diffusion of tracer in convection rolls.
\newblock {\em Phys. Fluids A}, 1(3):462--469, 1989.

\bibitem{Young88}
W.~R. Young.
\newblock Arrested shear dispersion and other models of anomalous diffusion.
\newblock {\em J. Fluid Mech.}, 193:129--149, Aug 1988.

\bibitem{Young10}
W.~R. Young.
\newblock Private communication, 2010.

\end{thebibliography}
  
\end{document}